\documentclass{amsart}

\usepackage{amsmath}
\usepackage{amssymb}
\usepackage{amsthm}
\usepackage{stmaryrd}
\usepackage{pgfplots}
\usepackage{pgfplotstable}
\usepackage{tikz}
\usepackage{algorithm}
\usepackage{xcolor}
\usepackage{algorithmicx}
\usepackage{algpseudocode}
\usepackage{todonotes}

\newtheorem{thm}{Theorem}
\newtheorem{lem}{Lemma}

\theoremstyle{definition}
\newtheorem{rem}{Remark}
\newtheorem{algo}{Algorithm}

\newcommand{\Div}[0]{\hspace{0.15mm}\mathrm{div}\,}
\newcommand{\Curl}[0]{\mathbf{curl}\,}
\newcommand{\bs}[1]{\boldsymbol{#1}}

\newcommand{\Lamspace}[0]{\bs{\Lambda}}
\newcommand{\Q}{\bs{Q}}

\newcommand{\Lam}[0]{\bs{\lambda}}
\newcommand{\Mu}[0]{\bs{\mu}}
\newcommand{\X}[0]{\bs{\xi}}
\newcommand{\jump}[1]{\left\llbracket #1 \right\rrbracket}
\newcommand{\Bf}[0]{\mathcal{B}}

\newcommand{\N}[0]{\bs{n}}

\newcommand{\Th}[0]{\mathcal{T}_h}
\newcommand{\Eh}[0]{\mathcal{E}_h}

\newcommand{\enorm}[1]{{\left\vert\kern-0.25ex\left\vert\kern-0.25ex\left\vert #1 
    \right\vert\kern-0.25ex\right\vert\kern-0.25ex\right\vert}}
\newcommand{\dx}[0]{\,\mathrm{d}x}
\newcommand{\ds}[0]{\,\mathrm{d}s}
\newcommand{\con}{\operatorname{con}, T}
\newcommand{\osc}{\operatorname{osc}}
\newcommand{\Cl}{\Pi_h}
\newcommand{\PS}{\Pi_0}

\newcommand{\IL}{I_{\mathcal{L}}}
\newcommand{\id}{\operatorname{id}}

\renewcommand{\P}{\mathbb{P}}
\newcommand{\rr}{\mathbb{R}}

\begin{document}

\author[Tom Gustafsson]{Tom~Gustafsson}
\author[Philip L. Lederer]{Philip~L.~Lederer}

\title{Mixed finite elements for Bingham flow in a pipe}

\begin{abstract}
  We consider mixed finite element approximations of viscous, plastic
  Bingham flow in a cylindrical pipe. A novel a priori and a posteriori error
  analysis is introduced which is based on a discrete mesh dependent
  norm for the normalized Lagrange multiplier. This allows proving 
  stability for various conforming finite elements. Numerical examples
  are presented to support the theory and to demonstrate adaptive mesh refinement.
\end{abstract}

\maketitle

\section{Introduction}

In this work we consider a viscous, plastic (Bingham) fluid which
behaves like a solid at low stresses and like a viscous fluid at high
stresses, see \cite{
MR1810507,MOSOLOV1966841,MOSOLOV1967609,Huilgol1995OnTD} and the more
recent review article \cite{MR3205441}. An everyday example is
toothpaste which extrudes from the tube as a solid plug when stress is
applied, remains solid in the middle of the plug and exhibits
fluid-like behavior near the tube wall. The velocity stays constant
within the solid part, i.e.~$\nabla u = \bs{0}$, and this condition is
enforced using a normalized vectorial Lagrange multiplier $\Lam$. Note
that the physical stress vector is $g \Lam$, and the shear stress is
given by its length $g |\Lam|$. Here $g>0$ is a given fixed threshold
value for the shear stress at which the solid becomes liquid when
exceeded. According to Section 8 in \cite{MR0521262} this gives the
strong formulation 
\begin{subequations}  \label{eq:exactprob}
\begin{alignat}{2}
  -\mu \Delta u - g\,\Div \Lam &= f \quad && \text{in $\Omega$}, \label{eq:exactprob_one}\\
  \Lam \cdot \nabla u &= |\nabla u| \quad && \text{in $\Omega$,} \label{eq:exactprob_three}\\
  |\Lam | &\leq 1\quad && \text{in $\Omega$,} \label{eq:exactprob_three1}\\
  u&=0 \quad && \text{on $\partial \Omega$}, \label{eq:exactprob_two}
\end{alignat}
\end{subequations}
where $\mu$ is the viscosity of the considered fluid, and $f$ describes
the pressure drop along the pipe. Note that in practice the pressure
drop is often constant over the cross section. However, in this work
we assume that $f \in L^2(\Omega)$. 

Several contributions in the field of Bingham-type fluid
computations were made by Glowinski and collaborators;
cf.~\cite{MR737005, he2000steady, MR0520279}. In the latter a linear
approximation was introduced and a suboptimal a priori error estimate
for the velocity was given. An optimal linear convergence of some low
order (mixed) methods was discussed in the works
\cite{MR3475655,MR449119}. 
Glowinski also provided an exact solution for the model problem of a
circular domain with a constant load. Since even this simple geometry and loading leads to a solution that is only in $H^{5/2-\epsilon}(\Omega)$, $\epsilon > 0$, a higher regularity for general Bingham-type flows is unlikely. 
Due to this non-smooth nature of the problem, the use of adaptive error
control seems highly desirable, see for example \cite{MR1992792} and
the references therein, and more recently in \cite{MR3874780}.
In particular, we would like to highlight the work \cite{MR2855533}
where the authors introduced the same a posteriori estimator that we
derive in this work. More precisely they bound the velocity error in
terms of the load, the discrete velocity $u_h$ and the discrete
approximation of the Lagrange multiplier~$\Lam_h$,
\begin{align*}
  \| u - u_h \|_1 \lesssim \eta(f,u_h,\Lam_h),
\end{align*}
where $\eta$ is some estimator. Although their definition of $\eta$ is
reasonable, proper error control is not guaranteed since their
stability analysis does not include a bound for $\Lam_h$. The main
problem can be traced back to the lack of a Babu\v{s}ka--Brezzi
condition for the considered linear Lagrangian velocity space and the
space of element-wise vector-valued constants for the Lagrange
multiplier (which is the same discretization used in
\cite{MR0520279}). As a result discrete stability for both $u_h$ and
$\Lam_h$ is not present, i.e.~the estimator $\eta$ could be arbitrarily large. 

The main contribution of this work is a novel stability and error
analysis of a mixed finite element approximation of
\eqref{eq:exactprob}. For this we build upon the ideas from one of the
authors work \cite{MR3723328} on obstacle problems, and the
corresponding references therein. Our analysis is based on proving a
discrete Babu\v{s}ka--Brezzi condition using a mesh dependent norm;
cf.~\cite{MR1076437}. This allows us to consider various finite element
pairs suitable for approximating \eqref{eq:exactprob}. Beside
continuous and discrete stability (see Section~\ref{sec::contstab} and
Section~\ref{sec::FEM}) we derive an a priori error estimate and
discuss linear convergence for sufficiently regular solutions in
Section~\ref{sec::apriori}. Our approach then further allows deriving
a residual based a posteriori error estimator (see
Section~\ref{sec::aposteriori}) which is globally and locally
efficient. We want to emphasize that our analysis gives full control
for both the error of the velocity and the error of the (divergence of
the) Lagrange multiplier. We conclude the work in
Section~\ref{sec::num} where we give insight on how to solve the
discrete system and provide several numerical examples to validate our
analysis.

\section{Continuous stability} \label{sec::contstab}
The weak formulation of \eqref{eq:exactprob} finds $u \in V$ and $\Lam \in
\Lamspace$ such that
\begin{subequations} \label{eq::weakform}
\begin{alignat}{2}
  (\mu \nabla u, \nabla v) + (g \nabla v, \Lam) &= (f, v) \quad && \forall v \in V, \label{eq:first}\\
  (g\nabla u, \Mu - \Lam) &\leq 0 \quad && \forall \Mu \in \Lamspace, \label{eq:second}
\end{alignat}
\end{subequations}
where $V = H^1_0(\Omega)$ and $\Lamspace = \{ \Mu \in L^2(\Omega,
\rr^2) : |\Mu| \leq 1~\text{a.e.~in $\Omega$}\}$; see
also~\cite{he2000steady}. 
Combining \eqref{eq:first} and \eqref{eq:second} gives
\begin{equation}
  \label{eq:weak1}
  (\mu\nabla u, \nabla v) + (g\nabla v, \Lam) + (g\nabla u, \Mu - \Lam)  \leq (f, v)
\end{equation}
for every $(v,  \Mu) \in V \times \Lamspace$. Note that
the solution of \eqref{eq:weak1} is unique up to a divergence-free
component, i.e.~$\Lam + \X$ is also a solution if $\Div \X = 0$. 
For the stability analysis we choose
the standard $H^1$-norm for the space $V$, and the dual norm $\|
\Div(\cdot) \|_{-1}$ for $\Lamspace$. Note that the latter is strictly
speaking not a norm, but only a seminorm. Thus, all error estimates
for $\Lam$ from this work will not prove any convergence of the
corresponding approximation in a strong sense, but only show
convergence of its distributional divergence. 


To simplify the notation we will from now on set $g=\mu=1$. 
In the following we use the shorthand notation
\begin{equation}
  \Bf(w, \X; v, \Mu) = (\nabla w, \nabla v) + (\nabla v, \X) + (\nabla w, \Mu).
\end{equation}
Using Cauchy-Schwarz and the continuity of the duality pairing
\begin{align*}
  (\nabla v, \X) = \langle \Div \X, v\rangle_{-1} \le \| v \|_1 \| \Div \X \|_{-1} \quad \forall v \in V, \X \in \Q, 
\end{align*}
with $\Q = L^2(\Omega, \rr^2)$, one immediately sees that $\Bf$ is continuous, i.e.~we have 
 \begin{align}  \label{eq:Bcont}
  \Bf(w, \X; v, \Mu) \lesssim (\| w \|_1 + \| \Div \X  \|_{-1})(\| v \|_1 + \| \Div \Mu \|_{-1}).
 \end{align}
Throughout the paper we write $a \lesssim b$ (or $a \gtrsim b$)
if there exists a constant $C>0$, independent of the finite element mesh,
such that $a \leq C b$ (or $a \geq C b$).

\begin{thm} \label{th:contstab}
 For every $(w, \X) \in V \times \Lamspace$ there exists a function $r \in V$ such
  that
  \begin{equation}
    \label{eq:stab1}
    \Bf(w, \X; r, -\X) \gtrsim (\| w \|_1 + \| \Div \X \|_{-1})^2
  \end{equation}
  and
  \begin{equation}
    \label{eq:stab2}
    \|r\|_1 \lesssim \| w \|_1 + \| \Div \X \|_{-1}.
  \end{equation}
\end{thm}
\begin{proof}
  We have
  \begin{equation}
    \label{eq:stabpf0}
    \Bf(w, \X; w, -\X) = \| \nabla w \|_0^2.
  \end{equation}
  Moreover, let $q \in V$. Then
  \begin{equation}
    \label{eq:stabpf1}
    \Bf(w, \X; q, 0) = (\nabla w, \nabla q) + (\X, \nabla q).
  \end{equation}
  If $q$ is chosen as the solution to
  \begin{equation}
    \label{eq:stabpf2}
    (\nabla q, \nabla z) = (\X, \nabla z) \quad \forall z \in V,
  \end{equation}
  then testing with $z = q$ gives $(\X, \nabla q) = \| \nabla q
  \|_0^2$. By the definition of $\| \cdot \|_{-1}$ and Cauchy--Schwarz
  we have
  \begin{equation}
    \label{eq:stabpf3}
    \| \Div \X \|_{-1} = \sup_{v \in V} \frac{(\X, \nabla v)}{\|v\|_1} \leq \| \nabla q \|_0.
  \end{equation}
  Now  choose $r := w + q$. Combining \eqref{eq:stabpf0}, \eqref{eq:stabpf1} and \eqref{eq:stabpf3},
  and applying Cauchy--Schwarz and Young's inequalities on \eqref{eq:stabpf1}
  gives the first result \eqref{eq:stab1}. 

  For \eqref{eq:stab2} the triangle inequality gives $\| r \|_1 \le
  \|w\|_1 + \|q\|_1$. Using Friedrichs inequality we get
  \begin{align*}
    \| q \|^2_1 \lesssim \| \nabla q \|^2_0 = (\nabla q, \X) = \langle q, \Div \X \rangle_{-1} \le \| q \|_1 \| \Div \X \|_{-1},
  \end{align*}
  which concludes the proof. 
\end{proof}

\section{Finite element method}\label{sec::FEM}

Let $V_h \subset V$ and $\Q_h \subset \Q$.
We define the discrete subspace $\Lamspace_h \subset \Q_h$ as $\Lamspace_h = \{ \Mu_h \in \Q_h : | \Mu_h | \leq 1~\text{a.e.~in $\Omega$}\}$.
Let $\Th$ be a shape regular triangulation of $\Omega$
and $h_T$ denote the diameter of $T \in \Th$. Further let $\Eh$ denote the set of edges with length $h_E$ for all $E \in \Eh$, for which we have, due to shape regularity, $h_E \sim h_T$.
The discrete norm for $\Mu_h \in \Lamspace_h$ is
\begin{equation}  \label{eq:meshdepnorm}
  \| \Div \Mu_h \|_{-1,h}^2 = \sum_{T \in \Th} h_T^2 \| \Div \Mu_h \|_{0,T}^2 + \sum_{E \in \Eh} h_E \|\!\jump{\Mu_h \cdot \N}\!\|_{0,E}^2,
\end{equation}
where $\jump{\cdot}$ is the usual jump operator. 
The discrete formulation reads: find $(u_h, \Lam_h) \in V_h \times \Lamspace_h$
such that
\begin{equation} \label{eq:discproblem}
  \Bf(u_h, \Lam_h; v_h, \Mu_h - \Lam_h) \leq (f,v_h) \quad \forall (v_h, \Mu_h) \in V_h \times \Lamspace_h.
\end{equation}
As in the continuous setting, we can prove stability of the mixed method \eqref{eq:discproblem} 
if the following Babu\v{s}ka--Brezzi condition is valid
\begin{equation}
  \label{eq:disclbb}
  \sup_{v_h \in V_h} \frac{(\X_h, \nabla v_h)}{\|v_h\|_1} \gtrsim \| \Div \X_h \|_{-1} \quad \forall \X_h \in \Q_h.
\end{equation}

\begin{thm} \label{thm:discretestability}
  Suppose $V_h$ and $\Lamspace_h$ satisfy \eqref{eq:disclbb}. Then for every
  $(w_h, \X_h) \subset V_h \times \Q_h$ there exists a function $r_h \in V_h$ such
  that
  \begin{equation}
    \label{eq:stab1_disc}
    \Bf(w_h, \X_h; r_h, -\X_h) \gtrsim (\| w_h \|_1 + \| \Div \X_h \|_{-1})^2,
  \end{equation}
  and
  \begin{equation}
    \label{eq:stab2_disc}
    \|r_h\|_1 \lesssim \| w_h \|_1 + \| \Div \X_h \|_{-1}.
  \end{equation}
\end{thm}
\begin{proof}
  This is similar to the proof of
  Theorem~\ref{th:contstab} but using \eqref{eq:disclbb} in the
  intermediate step \eqref{eq:stabpf3}.
\end{proof}
An explicit proof of condition \eqref{eq:disclbb} might be difficult depending on 
the choice of $V_h$ and $\Q_h$. To this end the following theorem shows that
it is sufficient to prove a discrete condition using the mesh dependent norm \eqref{eq:meshdepnorm}.
\begin{thm} \label{thm:discretelbb}
  If the discrete Babu\v{s}ka--Brezzi condition 
  \begin{align}
  \label{eq:disclbb_meshnorm}  
    \sup_{v_h \in V_h} \frac{(\X_h, \nabla v_h)}{\|v_h\|_1} \gtrsim \| \Div \X_h \|_{-1,h} \quad \forall \X_h \in \Q_h,
  \end{align}
  holds true, then the discrete spaces also fulfill \eqref{eq:disclbb}.
\end{thm}
\begin{proof}
  Let $\Cl: L^2(\Omega) \rightarrow
  V_h$ be the Cl\'ement quasi interpolation operator \cite{MR0400739} with the
   stability and interpolation properties
  \begin{subequations} \label{eq:clement}
  \begin{align}
  \| \Cl v \|_1 &\le C_s \| v \|_1 \quad \forall  v \in V, \\
    \Big (\sum_{T \in \Th} h_T^{-2} \| v - \Cl v\|_{0,T}^2 + \sum_{E \in \Eh} h_E^{-1} \| v - \Cl v\|_{0,E}^2 \Big)^{1/2} & \le C_i \| v \|_1 \quad \forall v \in V.
  \end{align}
\end{subequations}
First we show the existence of constants $C_1$ and $C_2$ such that
  \begin{align} \label{eq:toprove}
    \sup\limits_{v_h \in V_h} \frac{(\X_h,\nabla v_h)}{\| v_h \|_1} \ge C_1 \| \Div \X_h \|_{-1} - C_2 \| \Div \X_h \|_{-1, h}.
  \end{align}
For this we choose an arbitrary $\X_h \in \Lamspace_h$. 
  By Theorem~\ref{th:contstab} there exists a function $r \in V$ such that 
  \begin{align*}
    (\X_h,\nabla r) \ge C' \| r\|_1 \| \Div \X_h \|_{-1}. 
  \end{align*}
  Using the Cl\'ement operator we have for the difference 
  \begin{align*}
    (\X_h,\nabla (\Cl r - r)) 
    =& \sum_{T \in \Th} (\X_h,\nabla (\Cl r - r))_T  \\
    =& - \sum_{T \in \Th} (\Div \X_h, \Cl r - r )_T + (\Cl r - r, \X_h \cdot \N)_{\partial T}  \\
    \ge& - \sum_{T \in \Th} h_T^{-1}\|r - \Cl r\|_{0,T} h_T\|\Div \X_h\|_{0,T}\\
    &-  \sum_{E \in \Eh}h_E \| \jump{ \X_h \cdot \N } \|_{0,E} h_E^{-1}\| r - \Cl r \|_{0,E} \\
    \ge&  -2 C_i \| r \|_1 \| \Div \X_h \|_{-1, h}.
  \end{align*}
  Thus, in total we have 
  \begin{align*}
    (\X_h,\nabla \Cl r) &= (\X_h,\nabla (\Cl r -  r)) + (\X_h,\nabla r) \\
    & \ge - 2C_i \| r \|_1 \| \Div \X_h \|_{-1, h} + C' \| r\|_1 \| \Div \X_h \|_{-1} \\
    & \ge - C_2  \| \Div \X_h \|_{-1, h} \| \Cl r \|_1 + C_1 \| \Div \X_h \|_{-1} \| \Cl r \|_1,
  \end{align*}
 where $C_2 = 2C_i/C_s$ and $C_1 = C'/C_s$, which proves \eqref{eq:toprove}. 

  Now suppose that \eqref{eq:disclbb_meshnorm} is valid with a constant
  $C_3 > 0$, then we have for a convex combination with $t > 0$ 
  \begin{align*}
    \sup_{v_h \in V_h} \frac{(\X_h, \nabla v_h)}{\|v_h\|_1} 
    & = t \sup_{v_h \in V_h} \frac{(\X_h, \nabla v_h)}{\|v_h\|_1} 
    + (1 - t) \sup_{v_h \in V_h} \frac{(\X_h, \nabla v_h)}{\|v_h\|_1} \\
    & \gtrsim t  (C_1 \| \Div \X_h \|_{-1} - C_2  \| \Div \X_h \|_{-1, h}) + (1-t) C_3 \| \Div \X_h \|_{-1, h} \\
    & \gtrsim \tilde{C} \| \Div \X_h \|_{-1},
  \end{align*}
  where we have chosen $t := \frac{1}{2} C_3(C_2 + C_3)^{-1}$ and thus
  $\tilde{C} = C_1C_3/(2 (C_2 + C_3))$. 
\end{proof}

\subsection{Some stable discretizations} \label{sec:stablepairs}
Theorem~\ref{thm:discretelbb} shows that it is sufficient to prove condition 
\eqref{eq:disclbb_meshnorm} for some finite element spaces $V_h$ and
$\Q_h$. In the following we discuss some stable
choices. 
Let $\P^l(K)$
denote the space of polynomials of order $l \ge 0$ on $K \in \Th$, and let
$\P^l(K, \rr^2)$ denote its vector-valued version. 
The same notation is used for polynomals on $E \in \Eh$.
\subsubsection*{The $P^{k}P^{k-2}$ family}   
We choose the spaces 
\begin{subequations} \label{eq:p2p0spaces}
  \begin{align}
    V_h &:= \{v_h \in V: v_h|_T \in \P^k(T)~\forall T \in \Th \},\\
    \Q_h &:= \{\Mu_h \in \Q: v_h|_T \in \P^{k-2}(T, \rr^2)~\forall T \in \Th \}.
  \end{align}    
\end{subequations}
\begin{lem}
\label{lem:p2p0}
  The discrete spaces defined by \eqref{eq:p2p0spaces} fulfill the
  discrete condition \eqref{eq:disclbb_meshnorm}.
\end{lem}
\begin{proof}
  Let $\X_h \in \Q_h$ be arbitrary. We choose $v_h \in V_h$ such that
  \begin{subequations}\label{eq::pkpkmoments}
  \begin{alignat}{2}
    v_h(x_i) &=0 &\quad &\forall \textrm{~vertices~} x_i, \label{eq::pkpkmomentsverts} \\ 
    \int_E v_h r \ds &=\int_E h_E \jump{\X_h \cdot \N} r \ds  &\quad & \forall r \in \P^{k-2}(E), \forall E \in \Eh, \\
    \int_T v_h l \dx &=-\int_T h_T^2 \Div \X_h l \dx  &\quad & \forall l \in \P^{k-3}(T), \forall T \in \Th.
  \end{alignat}
    \end{subequations}
  Element-wise integration by parts and using that $\jump{\X_h \cdot \N} \in \P^{k-2}(E)$ for all edges $E \in \Eh$, and $\Div \X_h \in \P^{k-3}(T)$ for all elements $T \in \Th$, we have 
  \begin{align*}
    (\nabla v_h, \X_h)  
    = \| \Div \X_h \|_{-1,h}^2.
  \end{align*}
  By a standard scaling argument we have on each element $T \in \Th$
  \begin{align} \label{eq::dicnormequi}
  \begin{aligned}
      \| \nabla v_h \|^2_{0,T} 
      &\sim 
      \sum\limits_{E \subset \partial T} \frac{1}{h_E} \| v_h \|^2_{0,E} + \frac{1}{h_T^2}\| v_h \|^2_{0,T} \\
      &\sim 
      \sum\limits_{E \subset \partial T} \frac{1}{h_E} \| \Pi^{k-2}_E v_h \|^2_{0,E} + \frac{1}{h_T^2}\| \Pi^{k-3}_T v_h \|^2_{0,T}, 
      \end{aligned}
  \end{align}
  where $\Pi^{k-2}_E$ and $\Pi^{k-3}_T$ are the edge-wise and element-wise $L^2$-projection onto polynomials of order $k-2$ and $k-3$, respectively. Note that the second equivalence follows due to $v_h$ vanishing at all vertices, see \eqref{eq::pkpkmomentsverts}. 
  Using this equivalence we have by the moments \eqref{eq::pkpkmoments} and the Cauchy--Schwarz inequality
  \begin{align*}
    \| \nabla v_h \|^2_{0,T}  
     \lesssim & 
    \sum\limits_{E \subset \partial T} \frac{1}{h_E} \int_E (\Pi^{k-2}_E v_h)^2 \ds + \frac{1}{h_T^2} \int_T (\Pi^{k-3}_T v_h)^2 \dx \\
     =& \sum\limits_{E \subset \partial T} \int_E \Pi^{k-2}_E v_h \jump{\X_h \cdot \N} \ds - \int_T \Pi^{k-3}_T v_h \Div \X_h \dx \\
     \lesssim&
    \Big( \sum\limits_{E \subset \partial T} \frac{1}{h_E} \| \Pi^{k-2}_E v_h \|^2_{0,E} + \frac{1}{h_T^2}\| \Pi^{k-3}_T v_h \|^2_{0,T} \Big)^{1/2}\\
    &\Big( \sum\limits_{E \subset \partial T} {h_E} \| \jump{\X_h \cdot \N } \|^2_{0,E} + {h_T^2}\| \Div \X_h \|^2_{0,T} \Big)^{1/2}.
  \end{align*}
  Summing over all elements and using the norm equivalence \eqref{eq::dicnormequi} we conclude the proof.
\end{proof}

\subsubsection*{The MINI family}   
We choose the spaces 
\begin{subequations} \label{eq:MINIspaces}
  \begin{align}
    V_h &:= \{v_h \in V: v_h|_T \in \P^{k+2}(T)~\forall T \in \Th, v_h|_E \in \P^k(E)~\forall E \in \Eh \},\\
    \Q_h &:= \{\Mu_h \in \Q: v_h|_T \in \P^k(T, \rr^2)~\forall T \in \Th \} \cap H^1(\Omega, \rr^2).
  \end{align}    
\end{subequations}
Note that since now $\Q_h$ is a subset of $H^1(\Omega,\rr^2)$, the normal jumps in $\| \Div(\cdot) \|_{-1, h}$ vanish.  
\begin{lem}
  The discrete spaces defined by \eqref{eq:MINIspaces} fulfill the
  discrete condition \eqref{eq:disclbb_meshnorm}.
\end{lem}
\begin{proof}
  Let $\X_h \in \Q_h$ be arbitrary. 
  We choose $v_h \in V_h$ such that it vanishes at all vertices and all edges, i.e.~$v_h \in H^1_0(T)$ for all elements $T \in \Th$. In addition $v_h$  fulfills
  \begin{alignat*}{2}
    \int_T v_h l \dx &=-\int_T h_T^2 \Div \X_h l \dx  &\quad & \forall l \in \P^{k-1}(T), \forall T \in \Th.
  \end{alignat*}
  Using integration by parts and the fact that $\Div \X_h \in \P^{k-1}(T)$ for all elements $T \in \Th$, we have again
  \begin{align*}
    (\nabla v_h, \X_h) 
    = \| \Div \X_h \|_{-1, h}^2.
  \end{align*}
  With similar scaling arguments as in the proof of Lemma~\ref{lem:p2p0} we also
  have $\| \nabla v_h \|_0 \lesssim \| \Div \X_h \|_{-1,
  h}$, which concludes the proof.
\end{proof}


\rem{
\textit{The Crouzeix--Raviart method.}
The last method we want to mention is using a nonconforming
approximation of the velocity. We define the spaces
  \begin{alignat*}{2}
    V_h &:= \{v_h \in L^2(\Omega, \rr): &&v_h|_T \in \P^1(T)~\forall T \in \Th, \\
   &\qquad &&v_h~\textrm{is continuous and vanishes at midpoints} \\ 
   &\qquad &&\textrm{of interior and boundary edges, respectively}\},\\
    \Q_h &:= \{\Mu_h \in L^2(\Omega, \rr^2)&&: v_h|_T \in \P^0(T, \rr^2)~\forall T \in \Th \}.
  \end{alignat*}    
Since the degrees-of-freedom of the velocity are again associated to edges,
the stability analysis is similar as for the $P^kP^{k-1}$ method with $k=2$. Further note
that since $\nabla v_h \in \Q_h$ locally on each element, one can
reformulate the mixed method \eqref{eq:discproblem} as a primal method
(without the Lagrange multiplier $\Lam_h$) which is similar to the
nonconforming approximations from \cite{MR3475655} and (as explained
in \cite{MR3475655}) the method in \cite{MR449119}. Due to the
extensive analysis therein, we do not consider this method in the
present work, but want to mention that our techniques can be applied
accordingly.}

\section{A priori error analysis}\label{sec::apriori}

In this section we present an a priori error estimate and prove a
linear convergence for $H^2$-regular velocity solutions. This stands in
contrast to the suboptimal result $\mathcal{O}(h^{1/2})$ (for a
linear approximation) from \cite{MR737005,MR0520279} and is in
accordance to the linear convergence results from
\cite{MR3475655,MR449119}.
Although the analysis could be extended to provide a better rate for
smooth solutions, a higher regularity can not be expected for Bingham-type flows 
as discussed in the introduction.


\begin{thm} \label{thm:bestapprox}
  Let $(u_h, \Lam_h) \in V_h \times \Lamspace_h$ be the solution of
  \eqref{eq:discproblem}, then for any $(w_h, \X_h) \in V_h \times
  \Q_h$ it holds
  \begin{align*}
  \| u - u_h \|_1 &+ \| \Div (\Lam - \Lam_h) \|_{-1} \\
  &\lesssim \| u - w_h \|_1 + \| \Div(\Lam - \X_h) \|_{-1} + \sqrt{(\nabla u, \Lam - \X_h)}.
\end{align*}
\end{thm}

\begin{proof}
  Let $(w_h, \X_h) \in V_h \times \Q_h$ be arbitrary. By the
  discrete stability, see Theorem~\ref{thm:discretestability}, we find
  $v_h$ such that
  \begin{align*}
    \Big( \| u_h - w_h \|_1 + \| \Div(\Lam_h - \X_h) \|_{-1} \Big)^2 &\lesssim \Bf(u_h - w_h, \Lam_h - \X_h; v_h, \X_h - \Lam_h) \\
    &\leq (f, v_h) - \Bf(w_h, \X_h; v_h, \X_h - \Lam_h).
  \end{align*}
  Using the continuous problem \eqref{eq:weak1} we get
  \begin{align*}
    &(f, v_h) - \Bf(w_h, \X_h; v_h, \X_h - \Lam_h) \\
    &=(\nabla u, \Lam_h - \X_h) + \Bf(u - w_h, \Lam - \X_h; v_h, \X_h - \Lam_h) \\
    &\leq (\nabla u, \Lam - \X_h)  + \Bf(u - w_h, \Lam - \X_h; v_h, \X_h - \Lam_h),
  \end{align*}
  which concludes the proof with the continuity of $\Bf$, see \eqref{eq:Bcont}, and 
  \begin{align*}
    &\| u - u_h \|_1 + \| \Div(\Lam - \Lam_h) \|_{-1} \\
    &\le \| u - w_h \|_1 + \| \Div(\Lam - \X_h) \|_{-1} + \| u_h - w_h \|_1 + \| \Div(\Lam_h - \X_h) \|_{-1}.
  \end{align*}
\end{proof}

The following lemma shows that we can expect a linear convergence
whenever the solution is at least $H^2$-regular. 
Note that it is essential to bound the error just in terms of  
$\Div \Lam$, since our analysis does not provide any control 
for the divergence-free part of $\Lam$. According to \cite{MR737005} we further 
have for a convex domain and a smooth soulution the stability estimate 
\begin{align*} 
    |u|_2 + \| \Div \Lam \|_0 \lesssim \|f\|_{0}.
\end{align*}

\begin{lem} \label{lem::apriori}
  Let $\Omega$ be simply connected and convex.
  Choose $V_h$ and $\Q_h$ as in
  Section~\eqref{sec:stablepairs}, and let $(u_h, \Lam_h) \in V_h \times
  \Lamspace_h$ be the corresponding discrete solution. Further let $u \in
  V \cap H^2(\Omega)$  and $\Lam \in \Lamspace \cap H(\Div, \Omega)$.
  Then there holds
  \begin{align*}
    \| u - u_h \|_1 + \| \Div(\Lam - \Lam_h) \|_{-1} \lesssim h ( |u|_2 + \| \Div \Lam \|_0 )
    \lesssim h \| f \|_0,
  \end{align*}
  where $h = \max\limits_{T \in \Th} \operatorname{diam}({T})$.
\end{lem}
\begin{proof}
  We solve the Dirichlet problem: find $\theta \in H^1_0(\Omega)$ such
  that 
  \begin{align*}
    (\nabla \theta, \nabla v) = (\Div \Lam, v) \quad \forall v \in H_0^1(\Omega),
  \end{align*}
  for which we have due to the assumptions on the domain that $|
 \theta |_2 \lesssim \| \Div \Lam \|_0$. Further, since $\Lam -
  \nabla \theta$ is divergence free (by construction) Theorem 3.1 in
  \cite{girault2012finite} shows that there exists a $\phi \in
  H^1(\Omega)$ such that $\Lam = \nabla \theta + \Curl \phi$. 

  Let $w_h := \IL u$, where $\IL$ is the Lagrange interpolation
  operator onto $V_h$, then by the approximation properties of
  $\IL$ we have
  \begin{align*}
    \| u - w_h \|_1 \lesssim h | u |_2.
  \end{align*} 
  Using integration by parts and that $\Curl \nabla v = 0$ for all $v
  \in V$ we have 
  \begin{align*}
    \| \Div(\Lam - \X_h) \|_{-1} &= \sup_{v \in V} \frac{(\Lam - \X_h, \nabla v)}{\| v \|_1} 
    = \sup_{v \in V} \frac{(\Lam - \Curl \phi - \X_h, \nabla v)}{\| v \|_1}.
  \end{align*}
  For the $P^kP^{k-2}$ familiy we choose $\X_h = \PS \nabla
  \theta$, then 
  \begin{align*}
    (\Lam - \Curl \phi - \X_h, \nabla v) &= ( (\id - \PS) \nabla \theta , \nabla v) 
\lesssim h | \theta |_2 \| \nabla v \|_0 
\lesssim h \| \Div \Lam \|_0 \| \nabla v \|_0.
  \end{align*}
  It remains to bound the last term. For this note that since $\id -
 \PS$ is orthogonal on constants we have with similar steps as above
  \begin{align*}
    (\nabla u, \Lam - \X_h) = ( \nabla u, (\id - \PS) \nabla \theta) = ((\id - \PS) \nabla u, (\id - \PS) \nabla \theta) \lesssim h | u |_ 2 h \| \Div \Lam \|_0,
  \end{align*}
  from which we conclude the proof.

  For the MINI family, we use the same
  steps as above, but choose $\X_h = \Cl \nabla \theta$. 
  Using Theorem 2.6 from \cite{MR2240051} we get again the bound 
  \begin{align*}
    (\nabla u, \Lam - \X_h) \lesssim h | u |_ 2 h \| \Div \Lam \|_0,
  \end{align*}
  and we conclude the proof with Theorem~\ref{thm:bestapprox}. 
\end{proof}

\section{A posteriori error analysis}\label{sec::aposteriori}
Since a high regularity of the solution cannot be expected in general,
this section is dedicated to a posteriori error control, enabling the use of
adaptive mesh refinement. We define the local error estimators -- including the dependency on $g$ and $\mu$ to allow for a direct implementation -- as
\begin{align*}
  \eta^2_T &:= h^2_T \| \mu \Delta u_h + g\,\Div \Lam_h + f \|^2_{0,T} , \\
  \eta^2_E &:= h_E \| \jump{ (\mu \nabla u_h + g\Lam_h) \cdot \N} \|^2_{0,E},\\ 
  \eta^2_{\con} &:=  g\int_T (| \nabla u_h| - \Lam_h \cdot \nabla u_h) \dx,
\end{align*}
and the global estimator 
\begin{align*}
  \eta := \sqrt{ \sum_{T \in \Th} \eta_T^2 + \sum_{E \in \Eh} \eta_E^2 + \sum_{T \in \Th} \eta_{\con}^2}.
\end{align*}
The element and edge estimators $\eta_T$ and $\eta_E$, respectively, are standard 
residual estimators as known from the literature. The additional term $\eta_{\con}$ 
can be interpreted as a consistency estimator of equation \eqref{eq:exactprob_three}.
Further we want to emphasize that all estimators only depend on the distributional divergence of $\Lam_h$ for which we have discrete stability, see Theorem~\ref{thm:discretestability}. While this is clear for $\eta_T$ and $\eta_E$, through integration by parts this is also evident for $\eta_{\con}$.
\begin{thm}
There holds the a posteriori error estimate 
\begin{align*}
  \| u - u_h \|_1 + \| \Div(\Lam - \Lam_h) \|_{-1} \lesssim \eta. 
\end{align*}
\end{thm}
\begin{proof}
  Using the continuous stability we find $v \in V$ such that 
  \begin{align*}
    &(\| u - u_h \|_1 + \| \Div(\Lam - \Lam_h) \|_{-1})^2 \\
    &\lesssim \Bf(u - u_h,\Lam - \Lam_h; v,\Lam_h - \Lam ) \\
    &= (\nabla (u - u_h), \nabla v) + (\nabla v, \Lam - \Lam_h) + (\nabla(u - u_h), 
    \Lam_h - \Lam),
  \end{align*}
  and $\| v \|_1 \lesssim \| u - u_h \|_1 + \| \Div(\Lam - \Lam_h)
  \|_{-1}$. We continue with the first two terms. Using the Cl\'ement
  operator we have 
  \begin{align*}
    (\nabla u_h, \nabla \Cl v) + (\nabla \Cl v, \Lam_h) = (f, \Cl v) = (\nabla u, \nabla \Cl v) + (\nabla \Cl v, \Lam),
  \end{align*}
  and thus 
  \begin{align*}
    &(\nabla (u - u_h), \nabla v) + (\nabla v, \Lam - \Lam_h) \\
    &= (\nabla (u - u_h), \nabla (v - \Cl v) ) + (\nabla (v - \Cl v),  \Lam - \Lam_h) \\
    &= \sum_{T \in \Th} (\nabla (u - u_h), \nabla (v - \Cl v) )_T + (\nabla (v - \Cl v),  \Lam - \Lam_h)_T.
  \end{align*}
  Since $\Div(-\nabla u - \Lam) = f$,
  see \eqref{eq:exactprob_one}, and $f \in L^2(\Omega)$, we have that $\nabla
  u + \Lam \in H(\Div, \Omega)$, i.e.~it is normal continuous. By that
  we have with integration by parts on each element
  \begin{align*}
    &(\nabla (u - u_h), \nabla v) + (\nabla v, \Lam - \Lam_h) \\
    &= \sum_{T \in \Th} (f + \Delta u_h + \Div \Lam_h, v - \Cl v)_T + \sum_{E \in \Eh} (\jump{(\nabla u_h + \Lam_h) \cdot \N}, v - \Cl v)_E.
  \end{align*} 
  Using the properties of $\Cl$, 
 cf.~\eqref{eq:clement}, we finally arrive at
  \begin{align*}
    (\nabla (u - u_h), \nabla v) &+ (\nabla v, \Lam - \Lam_h) \lesssim \Big( \sum_{T \in \Th} \eta_T^2 + \sum_{E \in \Eh} \eta_E^2\Big)^{1/2} \| v \|_1.
  \end{align*}

  It remains to bound the other term. For this note that
  \eqref{eq:second} gives $(\nabla u, \Lam_h) \le (\nabla u, \Lam)$,
  and thus as $| \Lam| \le 1$,
  \begin{align*}
    (\nabla(u - u_h), \Lam_h - \Lam) \le  (\nabla u_h, \Lam - \Lam_h) 
    &\le \sum_{T \in \Th} \int_T (| \nabla u_h ||\Lam| - \nabla u_h \cdot \Lam_h) \dx \\
    &\le \sum_{T \in \Th} \int_T (| \nabla u_h | - \nabla u_h \cdot \Lam_h) \dx,
  \end{align*}
  which concludes the proof.
\end{proof}

Theorem~\ref{thm:localeff} and Lemma~\ref{lem:globaleff} below 
provide local and global efficiency estimates, respectively, for the residual based estimators $\eta_T$ and $\eta_E$.
The proofs follow with similar steps as in
\cite{MR3723328}, i.e.~we will provide all details of the local
efficiency but refer to \cite{MR3723328} for the proof of
Lemma~\ref{lem:globaleff}. Further note that similarly as in \cite{MR3723328} 
it is not possible to provide an upper bound for the consistency error $\eta_{\con}$.

For the efficiency estimates we need some additional notation. 
Let $\omega \subset \Omega$ be arbitrary then we define for all 
$\Mu \in \Lamspace$ the local dual norm by 
\begin{align*}
  \| \Div \Mu\|_{-1,\omega} := \sup\limits_{v \in H^1_0(\omega)} \frac{\langle v, \Div \Mu \rangle_{-1,\omega}}{\| v \|_{1,\omega}} = \sup\limits_{v \in H^1_0(\omega)} \frac{(\nabla v, \Mu)_{\omega}}{\| v \|_{1,\omega}}.
\end{align*}
The subset $\omega$ will be either an element $T \in \Th$ or $\omega_E$, where $\omega_E$ denotes the edge-patch for a given edge $E \in \Eh$. 
Finally, let $f_h := \Pi^q f$ be the element-wise $L^2$ projection onto $\P^q(K)$ where $q$ is
the polynomial order of the space $\Q_h$, and let
\begin{align*}
    \osc_T(f) &:= h_T \| f - f_h \|_{0,T} \quad \textrm{and} \quad
    \osc(f) := \Big( \sum_{T \in \Th} h^2_T \| f - f_h \|^2_{0,T} \Big)^{1/2}.
\end{align*} 

\begin{thm} \label{thm:localeff}
  Let $v_h \in V_h$ and $\Mu_h \in \Lamspace_h$ be arbitrary. There
  holds the local efficiency
  \begin{align*}
    h_T \| \Delta v_h + \Div \Mu_h + f \|_{0,T}  &\lesssim \| u - v_h \|_{1,T} + \| \Div(\Lam - \Mu_h)\|_{-1,T} + \osc_T(f), \\
    h_E^{1/2} \| \jump{ (\nabla v_h + \Mu_h) \cdot \N} \|_{0,E}^2  &\lesssim  \| u - v_h \|_{1,\omega_E} + \| \Div(\Lam - \Mu_h)\|_{-1,\omega_E} + \sum_{T \subset \omega_E} \osc_T(f).
  \end{align*}
\end{thm}
\begin{proof}
  The proof commences with the usual localizing technique by means of
  a element-wise qubic bubble function $b_T$. We define the localized error on $T$ by 
  \begin{align*}
    \delta_T|_T := h_T^2 b_T (\Delta v_h + \Div \Mu_h + f_h),
  \end{align*}
  and $\delta_T = 0$ on $\Omega \setminus T$. Since $b_T$ vanishes on
  the element boundary we have that $\delta_T \in V$. Using the norm
  equivalence for polynomial spaces we then have
  \begin{align*}
    &h_T^2 \| \Delta v_h + \Div \Mu_h + f_h \|^2_{0,T} \\
    &\lesssim h_T^2 \| b_T^{1/2} (\Delta v_h + \Div \Mu_h + f_h) \|^2_{0,T}\\
    &= ( \Delta v_h + \Div \Mu_h + f_h,  \delta_T)_T \\
    &= ( \Delta v_h + \Div \Mu_h,  \delta_T)_T + (f,  \delta_T)_T + (f_h - f,  \delta_T)_T \\
    &= ( \Delta v_h + \Div \Mu_h,  \delta_T)_T + (-\Delta u,  \delta_T)_T - \langle \Div \Lam,  \delta_T \rangle_{-1,T} + (f_h - f,  \delta_T)_T,
  \end{align*}
  and, with integration by parts also
  \begin{align} \label{eq:intermediate_vol}
    &h_T^2 \| \Delta v_h + \Div \Mu_h + f_h \|^2_{0,T} \nonumber  \\
    &\lesssim ( \nabla (v_h - u) , \nabla \delta_T)_T + \langle \Div (\Mu_h - \Lam),  \delta_T\rangle_{-1, T} + (f_h - f,  \delta_T)_T.
  \end{align}
  By the inverse inequality for polynomials we have 
  \begin{align} \label{eq::scaling_edge}
    \| \delta_T
  \|_{1, T} \lesssim h^{-1}_T \| \delta_T\|_{0, T}  \sim  h_T \| \Delta v_h + \Div \Mu_h + f_h \|_{0,T},
  \end{align}
  and thus, with Cauchy--Schwarz inequality we derive the
  first estimate with
  \begin{align*}
    &h_T^2 \| \Delta v_h + \Div \Mu_h + f_h \|^2_{0,T} \\
    &\lesssim \| u-v_h \|_{1,T} \| \delta_T\|_{1, T} \
    + \| \Div (\Lam - \Mu_h) \|_{-1,T} \| \delta_T \|_{1, T} + \osc_T(f) h_T^{-1} \| \delta_T\|_{0, T}.
  \end{align*}
  For the other term we proceed similarly. For this let $\delta_E :=
  b_E \mathcal{E}(\jump{ (\nabla v_h + \Mu_h) \cdot \N})$ where
  $\mathcal{E}$ is the well known extension operator onto
  $H^1_0(\omega_E)$, see \cite{braess}, and $b_E$ is the quadratic edge bubble. Scaling arguments and the Poincar\'e inequality
  give 
  \begin{align*}
    \| \jump{ (\nabla v_h + \Mu_h) \cdot \N} \|_{0,E} \sim h_E^{-1/2} \| \delta_E \|_{0,E} \sim h_E \| \delta_E \|_{1,\omega_E}.
  \end{align*}
With the same steps as for the volume term we derive the estimate
\begin{align} \label{eq:intermediate_edge}
\begin{aligned} 
  \| \jump{ (\nabla v_h + \Mu_h) \cdot \N} \|_{0,E}^2
  \lesssim& (\Delta v_h +  \Div \Mu_h + f, \delta_E)_{\omega_E} \\
  &+ (\nabla(v_h - u), \delta_E)_{\omega_E} 
  + \langle \Div (\Mu_h - \Lam),  \delta_E  \rangle_{-1, \omega_E},
\end{aligned}
\end{align}
from which we conclude the proofs using the Cauchy--Schwarz inequality,
the estimates of the volume term from before and
\eqref{eq::scaling_edge}.
\end{proof}

\begin{lem}  \label{lem:globaleff}
  Let $v_h \in V_h$ and $\Mu_h \in \Lamspace_h$ be arbitrary. There
  holds the global efficiency
  \begin{align*}
    \Big( \sum_{T \in Th} h^2_T \| \Delta v_h + \Div \Mu_h + f \|^2_{0,T} \Big)^{1/2}&\lesssim \| u - v_h \|_{1} + \| \Div(\Lam - \Mu_h)\|_{-1} + \osc(f), \\
    \Big( \sum_{E \in \Eh} h_E \| \jump{ (\nabla v_h + \Mu_h) \cdot \N} \|^2_{0,E}\Big)^{1/2} &\lesssim  \| u - v_h \|_{1} + \| \Div(\Lam - \Mu_h)\|_{-1} +  \osc(f).
  \end{align*}
\end{lem}
\begin{proof}
  The proof follows with the same steps as in \cite{MR3723328} using
  the intermediate estimates \eqref{eq:intermediate_edge} and
  \eqref{eq:intermediate_vol}.
\end{proof}
\section{Numerical examples} \label{sec::num}

We apply an iterative algorithm to approximate the solution of the discrete
problem~\eqref{eq:discproblem}.  It is based on a reformulation of the inequality constraint
\eqref{eq:second} as
\begin{equation}
  \label{eq:eqconstraint}
  \Lam - \bs{P}(\Lam + \rho \nabla u) = \bs{0}, \quad \rho > 0,
\end{equation}
where $\bs{P}(\Mu) = \tfrac{\Mu}{\max(1, |\Mu|)}$ scales any vectors of
$\mathbb{R}^2$ to maximum length one, cf.~\cite{MR737005, he2000steady}
for discussion on similar algorithms and proofs of their convergence.
The reformulation is based on the fact that $\Lam_h$ in
\[
    (\X_h - \Lam_h, \Mu_h - \Lam_h) \leq 0 \quad \forall \Mu_h \in \Lamspace_h,
\]
is the orthogonal projection of $\X_h \in \Q_h$ onto $\Lamspace_h$, and the orthogonal projection is alternatively characterized
by $\bs{P}$~\cite[Section 3]{he2000steady}.

\begin{algo}[Uzawa iteration]
  \label{alg:uzawa}
  Let $(u_h^0, \Lam_h^0) \in V_h \times \Lamspace_h$
  be an initial guess, $TOL$ a given tolerance and set $i=1$
  \begin{enumerate}
  \item Solve $u_h^i$ from $(\mu \nabla u_h^i, \nabla v_h) = (f, v_h) - g(\Lam_h^{i-1}, \nabla v_h)$ for every $v_h \in V_h$.
  \item Calculate $\Lam_h^i = \bs{P}(\Lam_h^{i-1} + \rho \pi_h \nabla u_h^i)$ where $\pi_h : \Q \rightarrow \bs{Q}_h$ is the $L^2$ projection onto $\bs{Q}_h$.
  \item Stop if $\|\nabla(u_h^{i} - u_h^{i-1})\|_0 / \| \nabla u_h^{i-1} \|_0< TOL$.  If not, increment
    $i$ and go to step (1).
  \end{enumerate}
\end{algo}

We first attempt to approximate an analytical solution on a circle
$\Omega = \{ (x,y) \in \mathbb{R}^2 : x^2 + y^2 < R^2 \}$ using
uniform mesh refinements;
see Figure~\ref{fig:meshes} for the sequence of meshes.
For constant loading $f$, the
coincidence set is a smaller circle with the radius $R_p = 2g/f$.  The
analytical solution reads $u(r) = \frac{R - r}{2}( \frac{f}{2}(R + r) - 2g)$
when $r > R_p$ and is equal to the constant $u(R_p)$ when $r \leq R_p$.
Substituting the above expression into the strong formulation \eqref{eq:exactprob_one}
we find also an analytical expression for the divergence of $\Lam$.

The error of the different components of the discrete norm are given in
Figure~\ref{fig:apriori} with $TOL = 10^{-7}$, $R=1$, $g=0.1$, $f=0.5$
and $\rho = 10$ in accordance to the suggestion in \cite[Remark 3]{he2000steady}.
We observe that for the MINI and $P^2P^0$ methods all components converge at least linearly
whereas for $P^3P^1$ method the $H^1$ seminorm of $u-u_h$ is approximately
$\mathcal{O}(h^{1.7})$ and the discrete norm of $\Lam - \Lam_h$ is
approximately $\mathcal{O}(h^{1.6})$, i.e.~less 
than the quadratic convergence order that interpolation 
estimates would imply for a completely smooth solution.

Next, our aim is to improve the convergence rate with respect to the
total number of degrees-of-freedom $N$ using mesh adaptivity.
We use an adaptive mesh sequence based on the a posteriori estimate
of Section~\ref{sec::aposteriori}.
An element-wise error estimator $E_T$ is given by
\[
E_T^2 = \eta_T^2 + \sum_{\substack{E \in \Eh, \\ E \cap \partial T \not= \emptyset}} (\tfrac12 \eta_E)^2 + \eta_{\con}^2,
\]
and we split $T^\prime \in \Th$ if
\[
E_{T^\prime} > 0.5 \max_{T \in \Th} E_{T}.
\]
The mesh is refined using the red-green-blue refinement strategy
and Laplacian smoothing is applied on the refined mesh
to improve its shape regularity.
Some examples from the sequence of adaptive meshes are
given in Figure~\ref{fig:adaptmeshes}.
A comparison of the error between the uniform and adaptive
mesh sequences is given in Figure~\ref{fig:aposteriori}.
In particular, we observe that while the convergence rate of the
error is ultimately dictated by the largest component of the discrete norm
(i.e.~$\|\Lam - \Lam_h\|_{-1,h}$),
there is a visible improvement
in all of the components and, as a conclusion, the quadratic
rate is recovered with respect to the number of degrees-of-freedom.

Finally, we consider an example
in a square domain $\Omega = (0,1)^2$
with $f=3.6$, $g=1.25$, $\rho=1.5$, and no analytical solution; cf.~\cite{MR2855533, saramito2001adaptive} for similar examples.
Some meshes from the adaptive sequence and the total error estimators are given in Figure~\ref{fig:square}.
The final discrete solution is depicted in Figure~\ref{fig:squareconv}.
As before, we observe the adaptive refinement
focusing on the interfaces between
the liquid and solid
regions.  Moreover, the estimators successfully
locate and refine the so-called
stagnating regions at
corners of the square.

\begin{rem}
We used a quadratic representation of the circle boundary
in order to neglect the effect of inexact geometry representation.
\end{rem}

\begin{rem}
We found the following equivalent form of the estimator $\eta_{\con}$
to be more robust against numerical tolerances in Algorithm~\ref{alg:uzawa}:
\[
g\int_T (| \nabla u_h| - \bs{P}(\Lam_h + \rho \pi_h \nabla u_h) \cdot \pi_h \nabla u_h) \dx.
\]
\end{rem}

\begin{rem}
We consider methods only up to a linear Lagrange multiplier because for a higher order
method, in general, $\Lam_h \not\in \Lamspace_h$ when using Algorithm~\ref{alg:uzawa}.
\end{rem}

\section*{Acknowledgements}

The numerical results were created using scikit-fem~\cite{gustafsson2020scikit}
which relies heavily on NumPy~\cite{harris2020array},
SciPy~\cite{virtanen2020scipy} and Matplotlib~\cite{hunter2007matplotlib}.
The work was supported by the Academy of Finland (decisions 324611 and
338341).

\begin{figure}
  \includegraphics[width=0.49\textwidth]{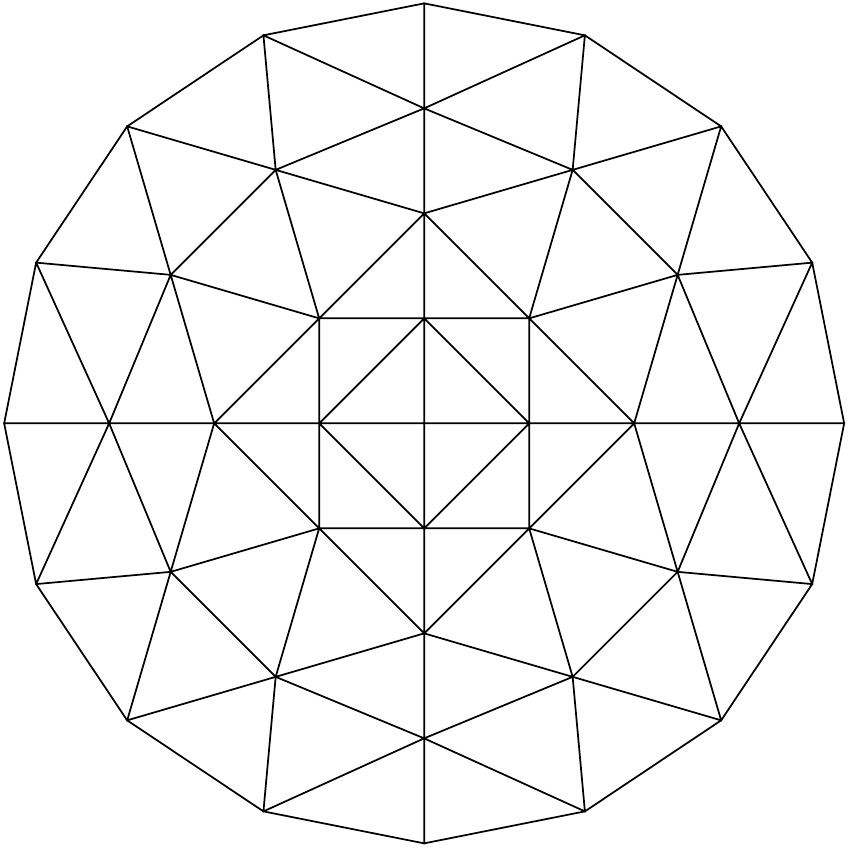}
  \includegraphics[width=0.49\textwidth]{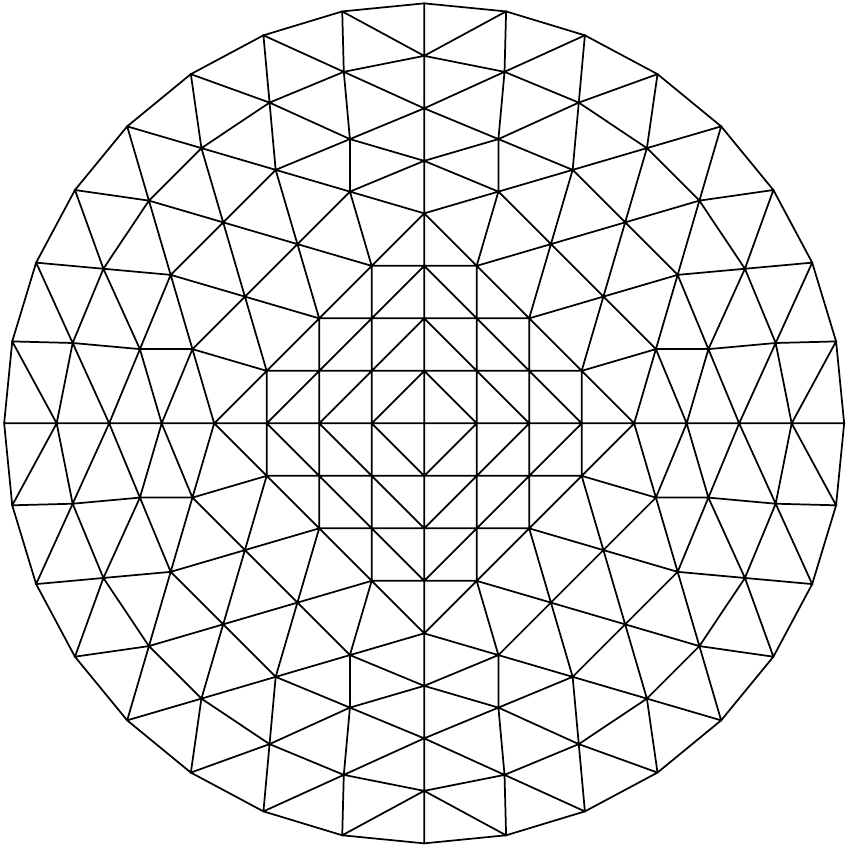}\\
  \includegraphics[width=0.49\textwidth]{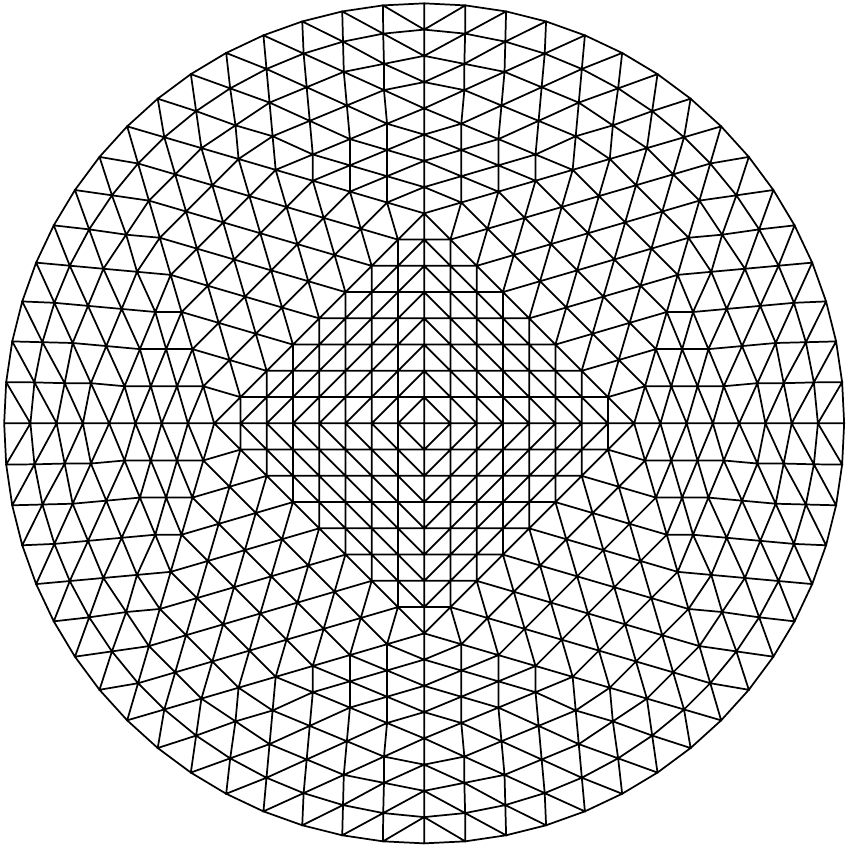}
  \includegraphics[width=0.49\textwidth]{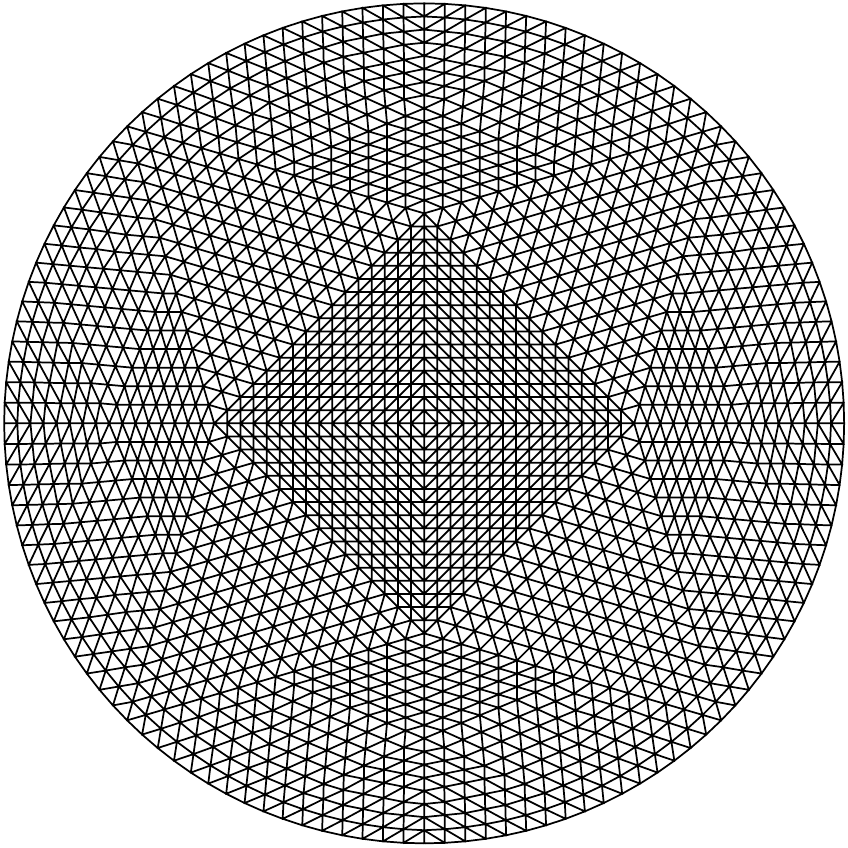}
  \caption{The sequence of uniformly refined meshes for the convergence study.}
  \label{fig:meshes}
\end{figure}

\begin{figure}
  \includegraphics[width=0.9\textwidth]{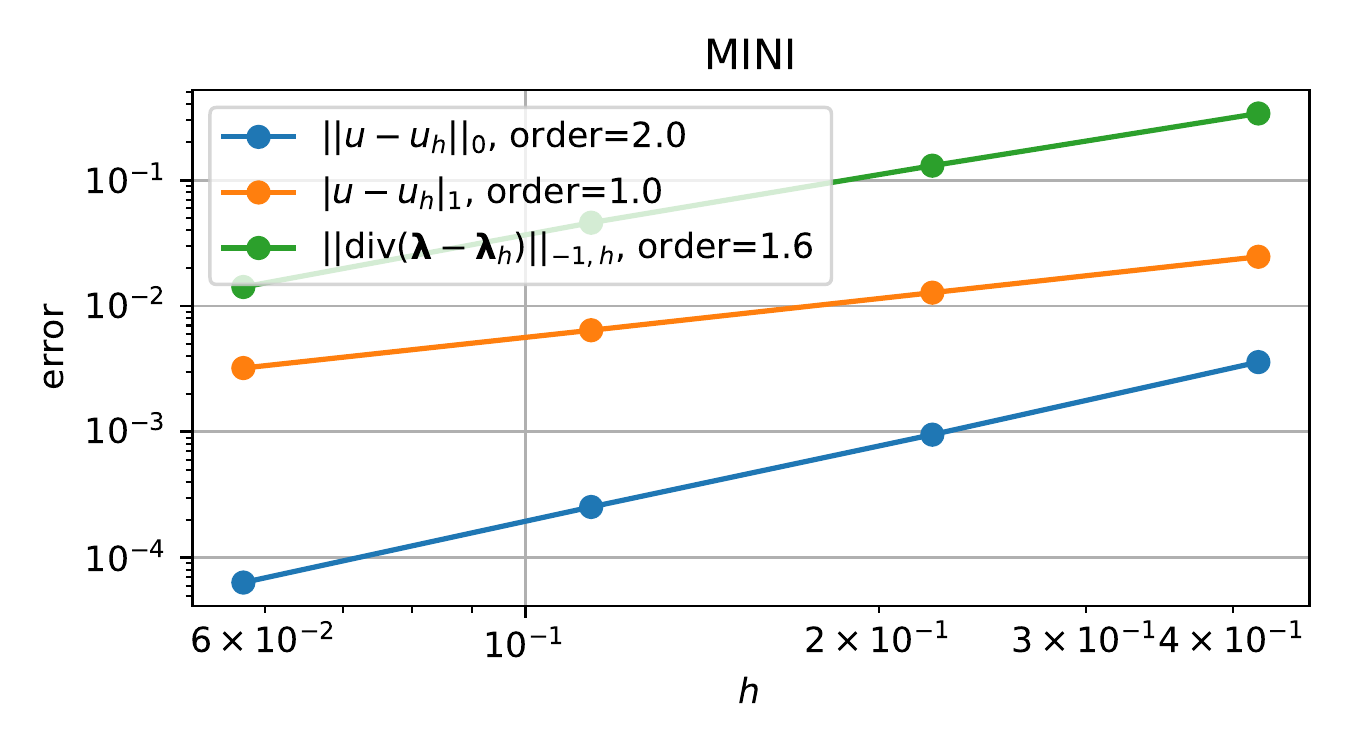}\\
  \includegraphics[width=0.9\textwidth]{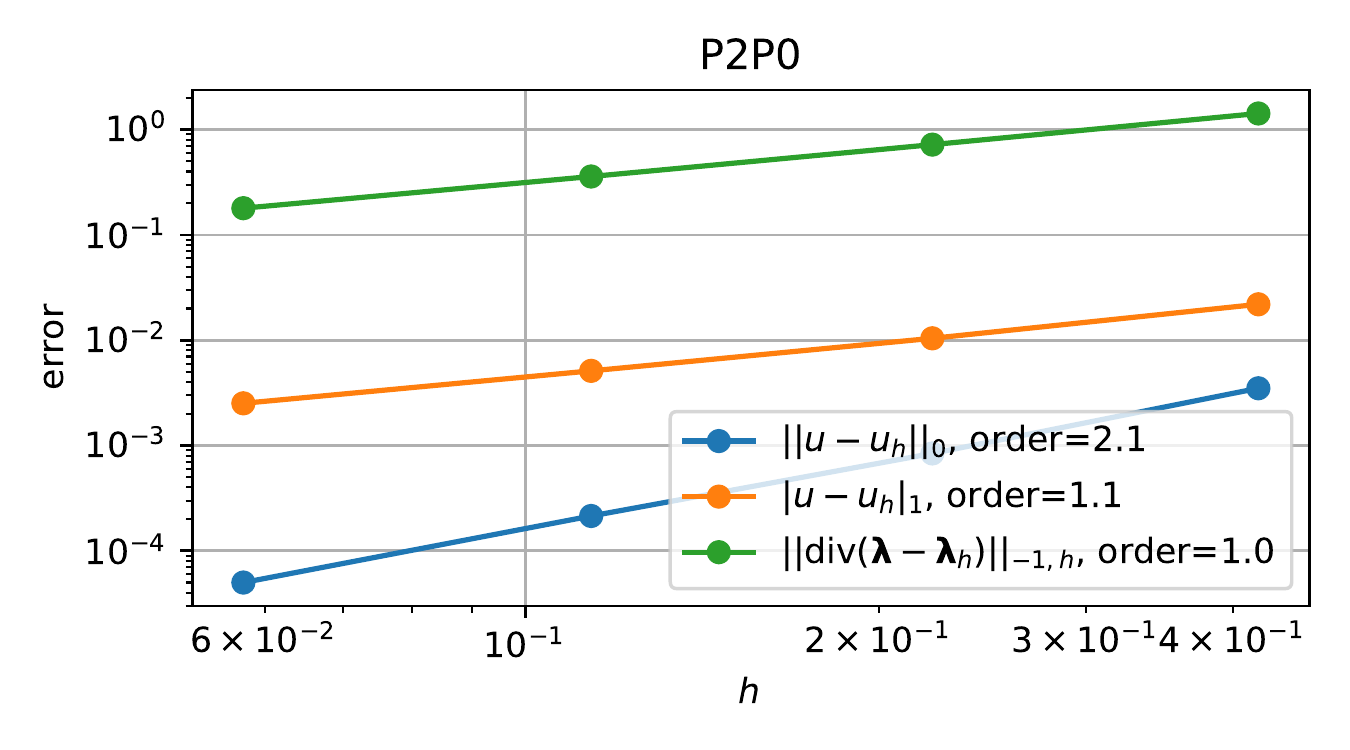}\\
  \includegraphics[width=0.9\textwidth]{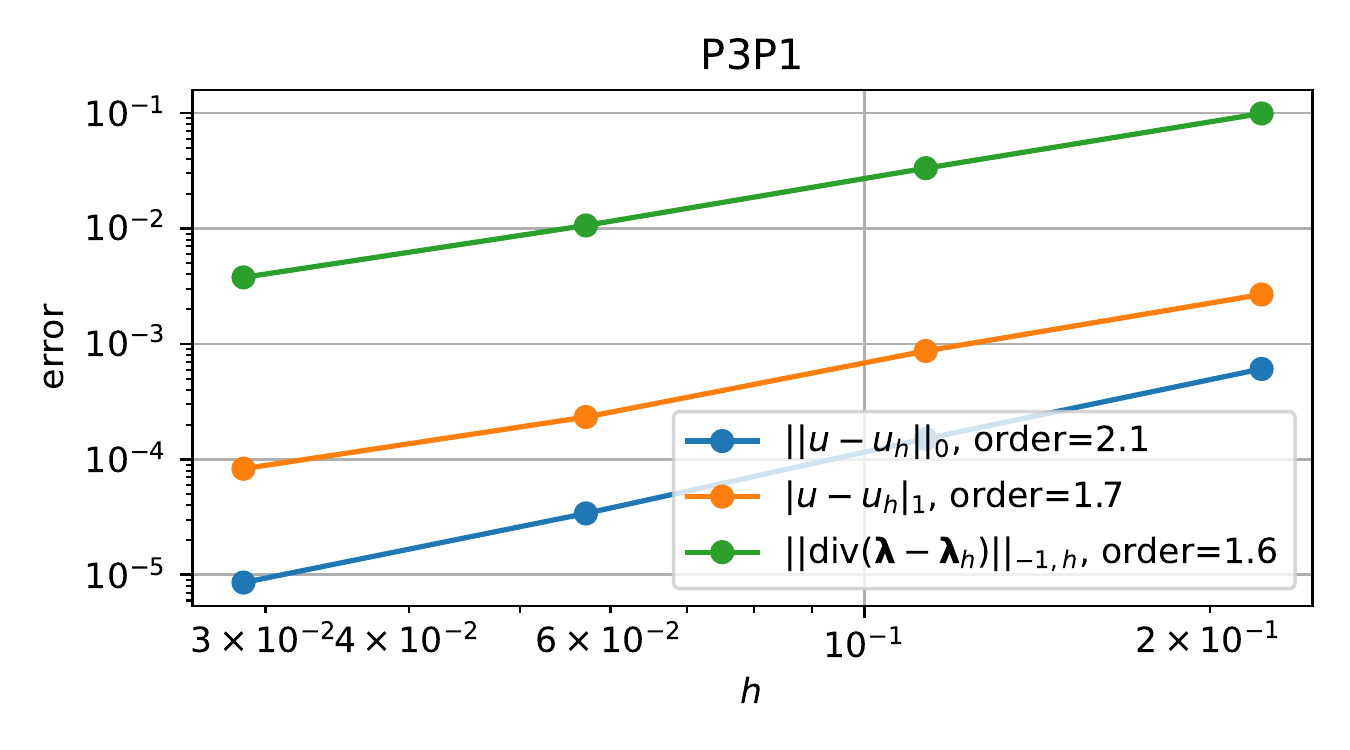}
  \caption{Error in the different components of the discrete norm for the
    circle problem as a function of the mesh parameter $h$ using the uniform
    mesh sequence.}
  \label{fig:apriori}
\end{figure}

\begin{figure}
  \includegraphics[width=0.49\textwidth]{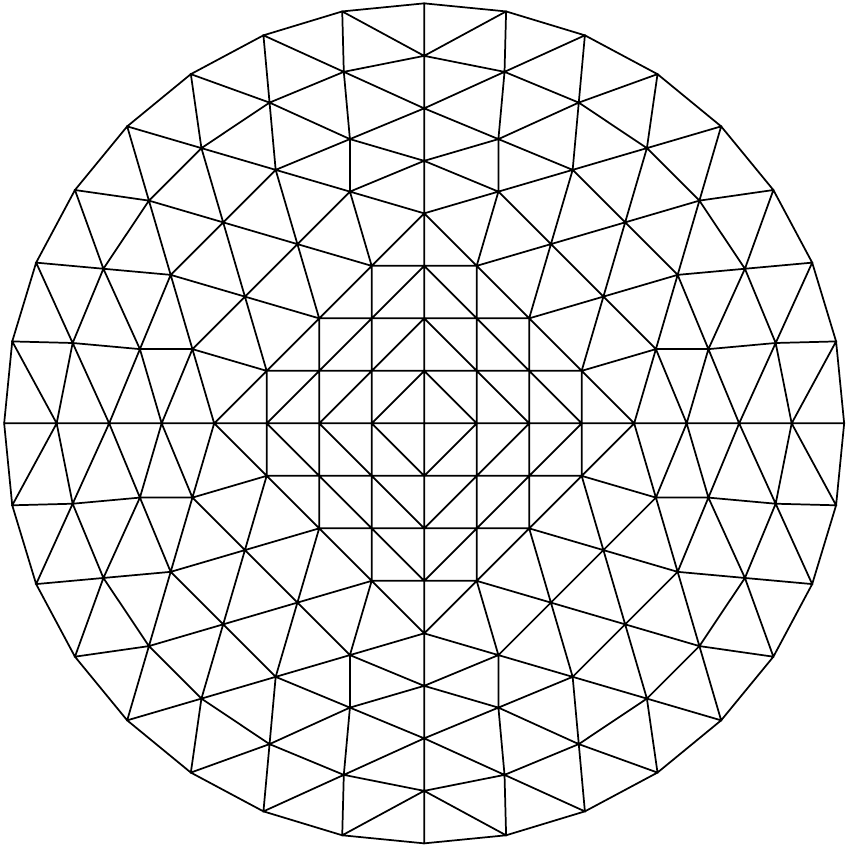}
  \includegraphics[width=0.49\textwidth]{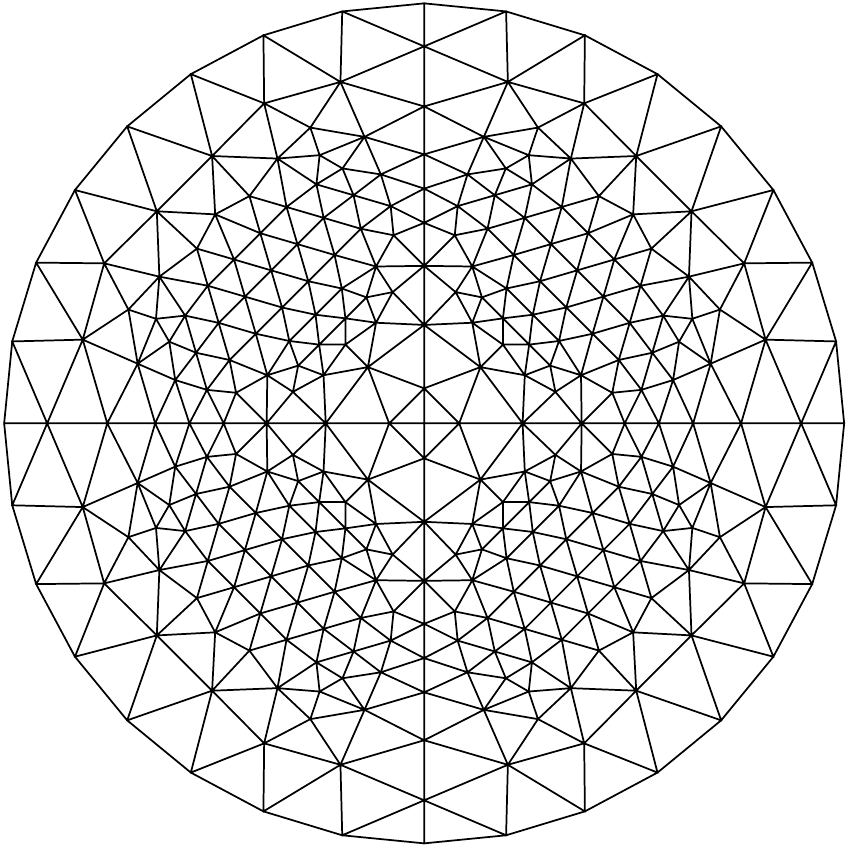}\\
  \includegraphics[width=0.49\textwidth]{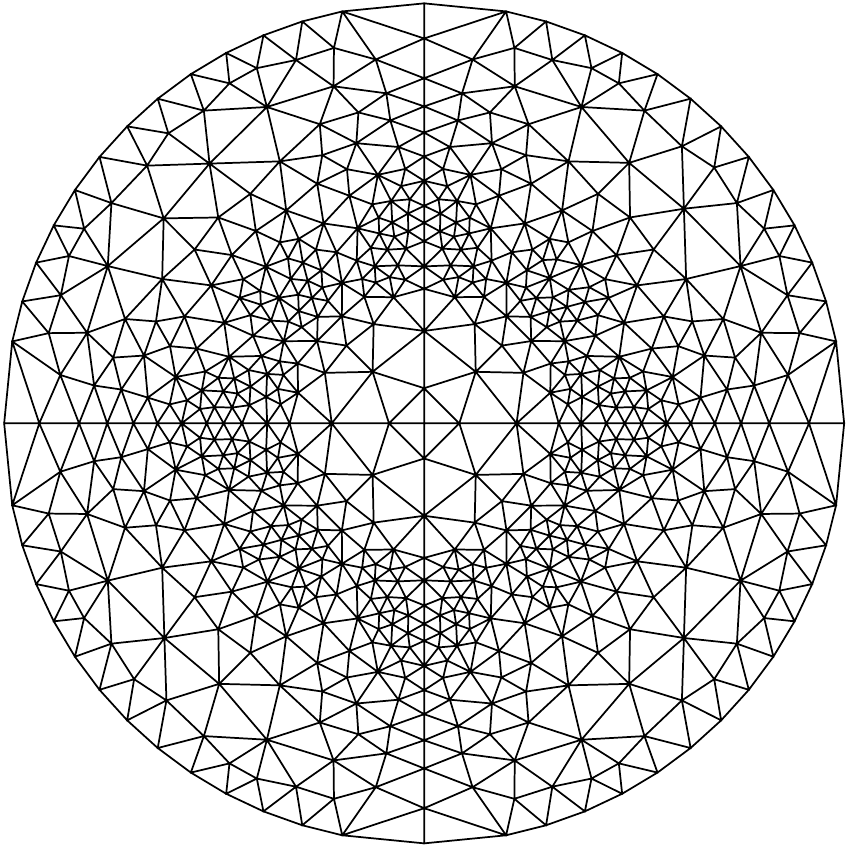}
  \includegraphics[width=0.49\textwidth]{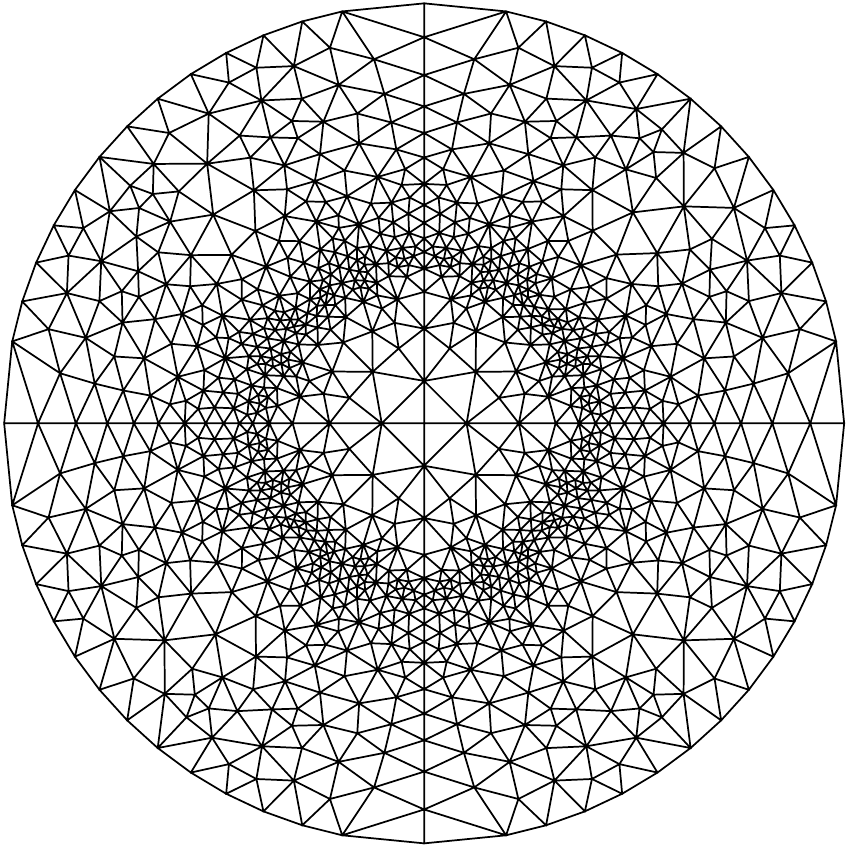}
  \caption{Some examples from the sequence of adaptively refined meshes for the circle problem.}
  \label{fig:adaptmeshes}
\end{figure}

\begin{figure}
  \includegraphics[width=\textwidth]{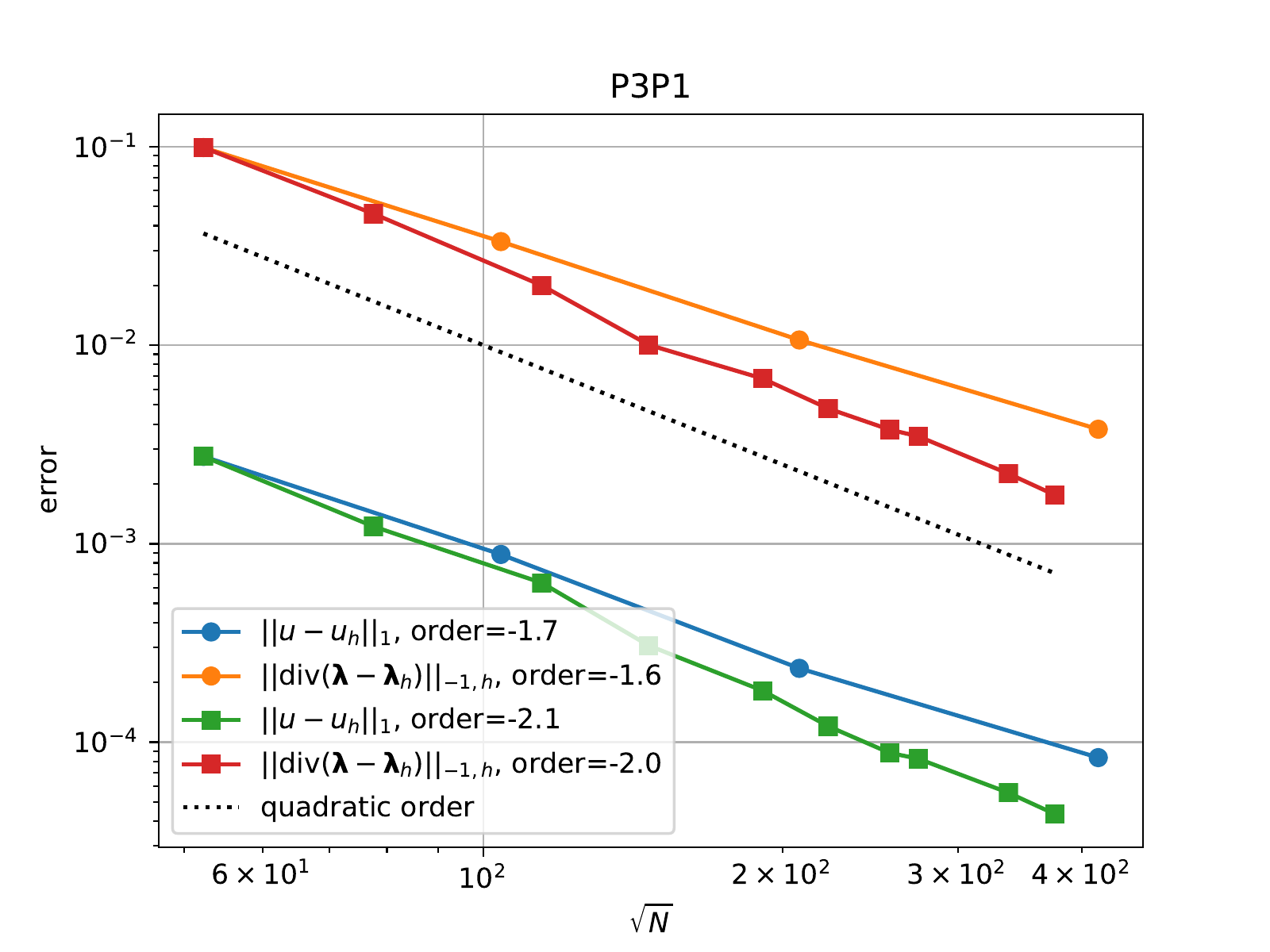}

  \caption{Error in the different components of the discrete norm for the
    circle problem using $P^3P^1$ method.  The horizontal axis is the square root of
    the total number of degrees-of-freedom $N$.  A comparison is made between
    the uniform mesh sequence (circles) and the adaptive mesh sequence
    (squares).}
  \label{fig:aposteriori}
\end{figure}

\begin{figure}
    \centering
    \includegraphics[width=0.4\textwidth]{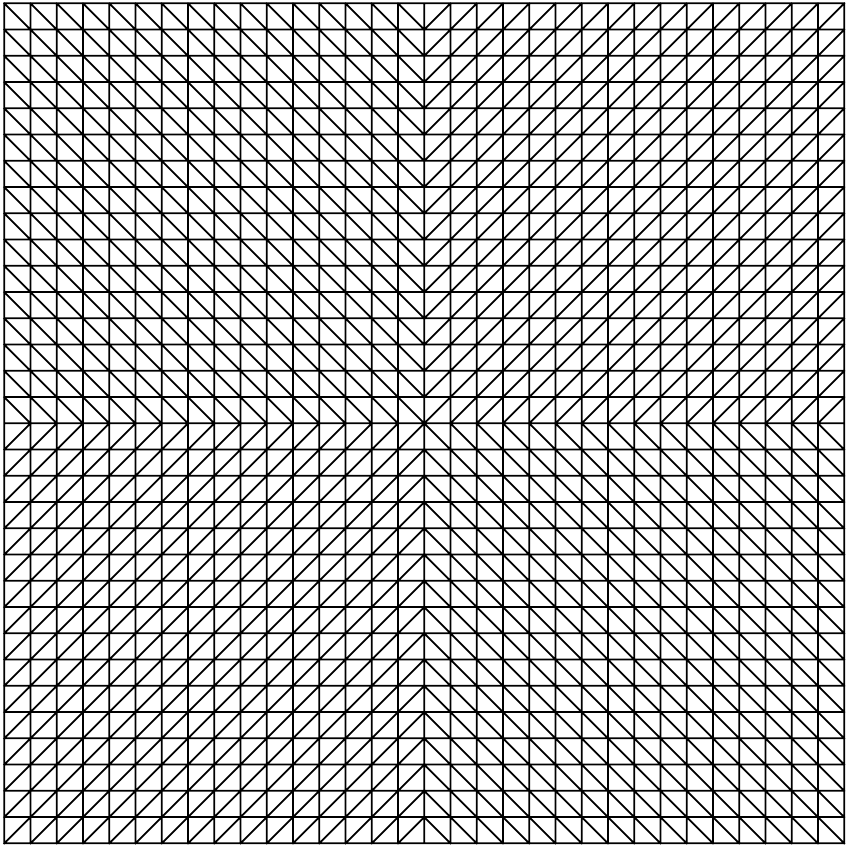}
    \hspace{0.5cm}
    \includegraphics[width=0.4\textwidth]{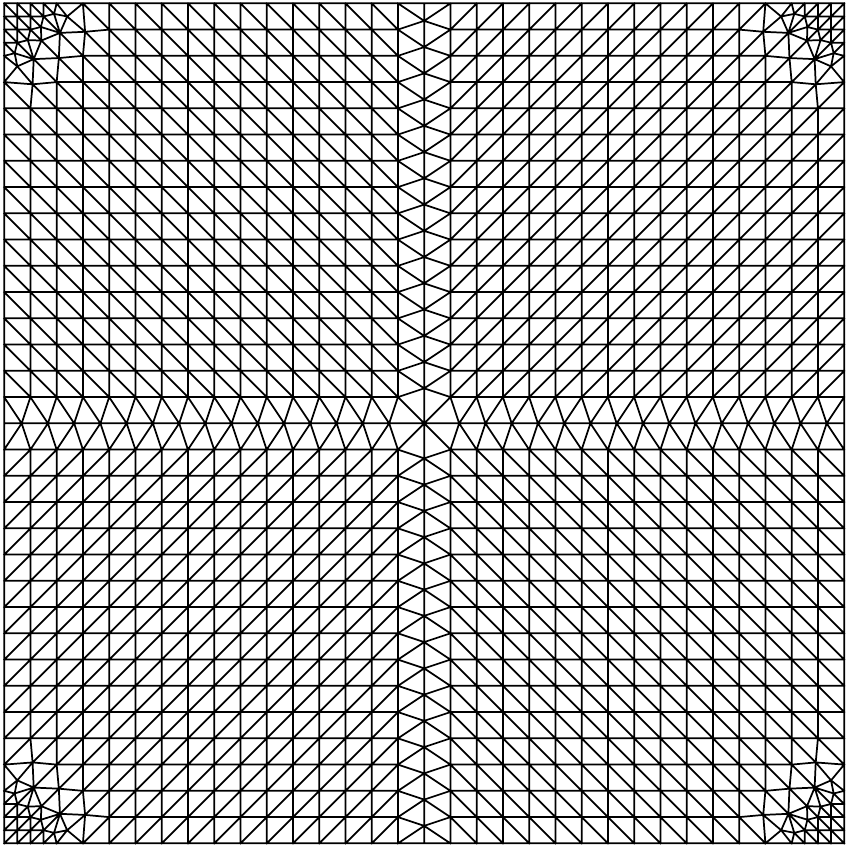}\\[0.6cm]
     \includegraphics[width=0.4\textwidth]{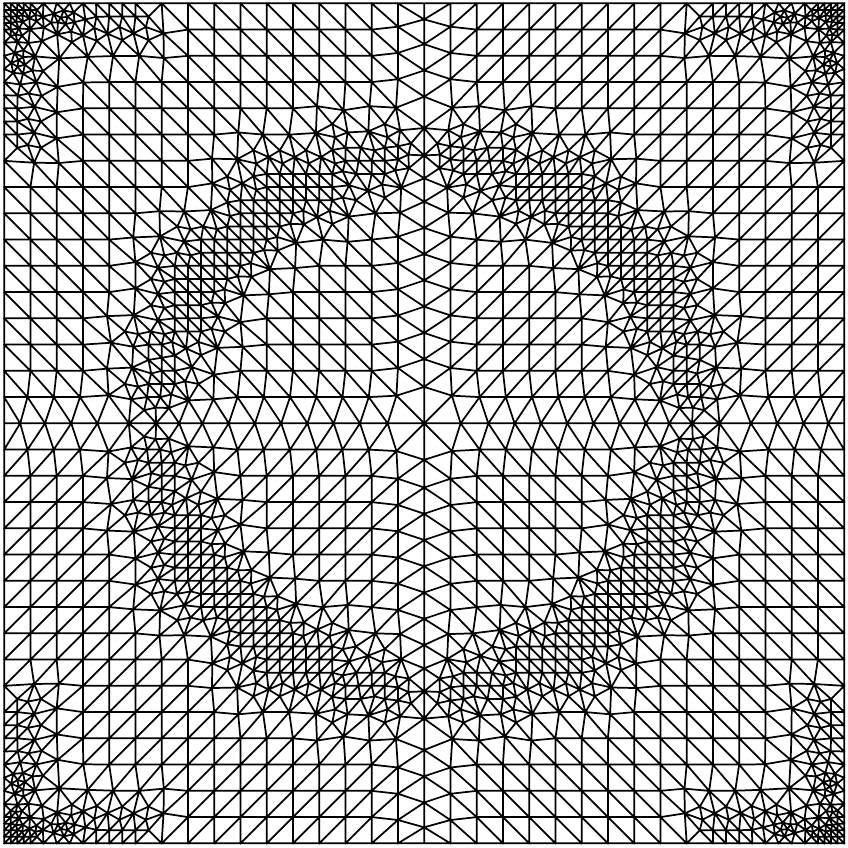}
     \hspace{0.5cm}
      \includegraphics[width=0.4\textwidth]{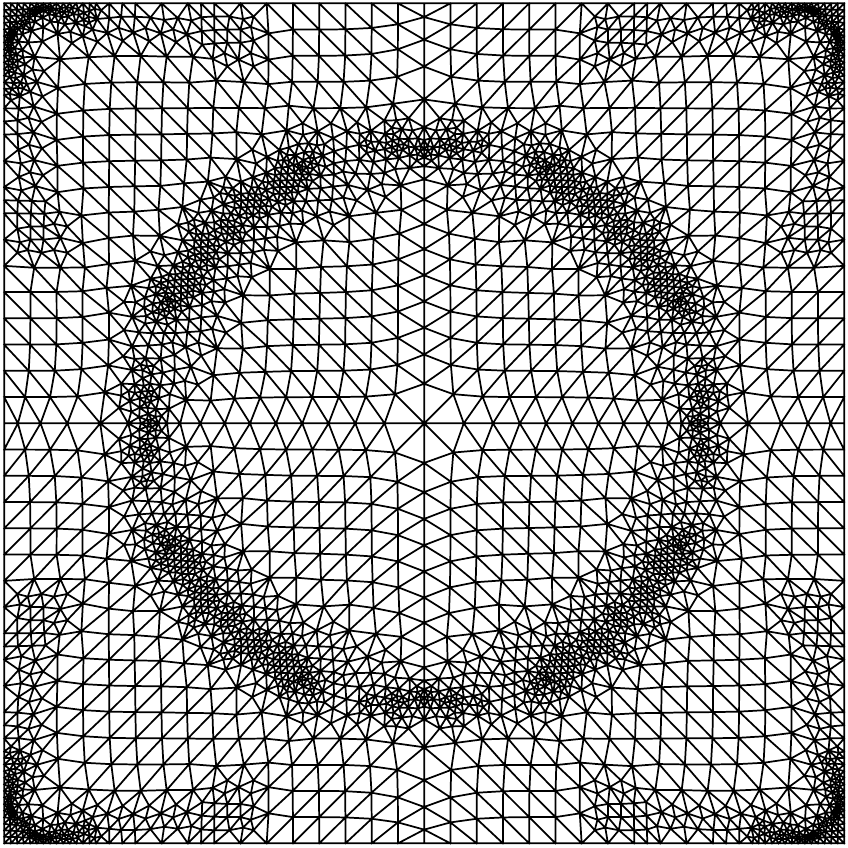}
      \\[0.3cm]
    \includegraphics[width=0.8\textwidth]{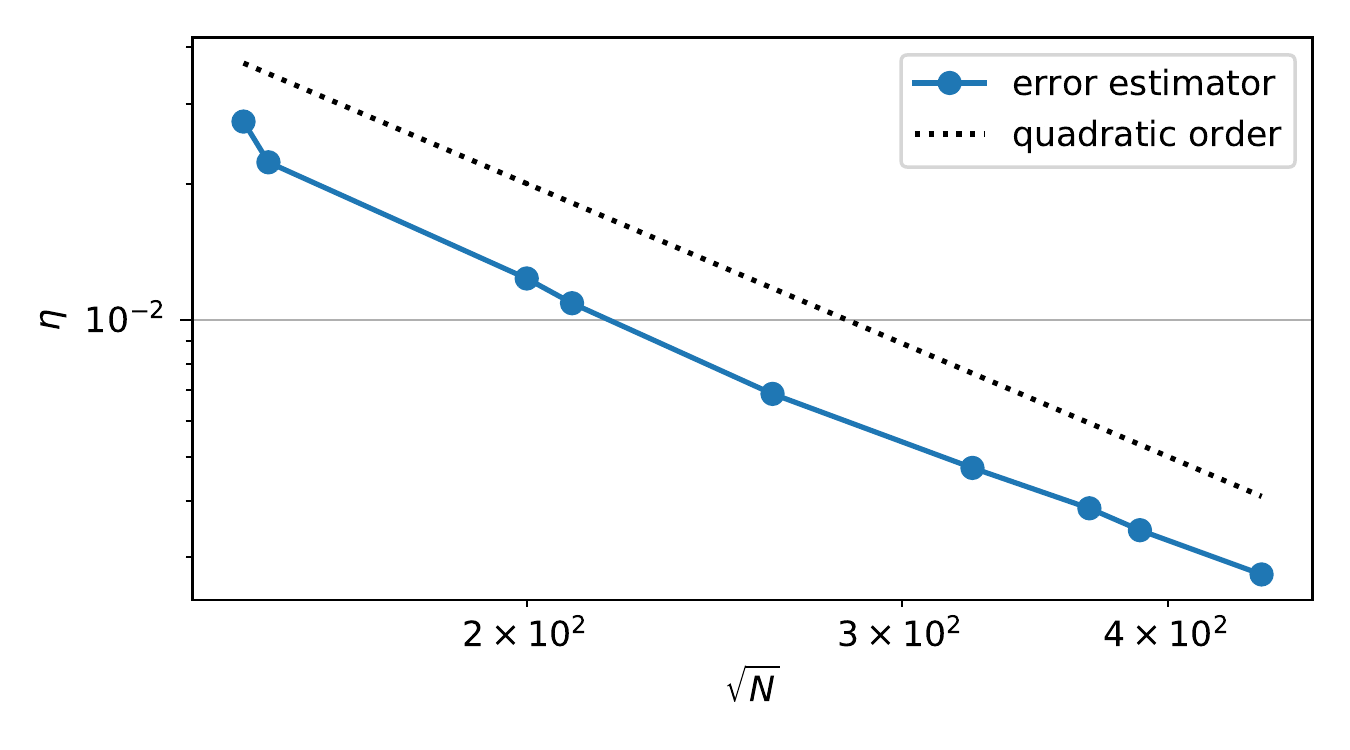}
    \caption{Some examples from the sequence of adaptively refined meshes for the square problem, and the total error estimator $\eta$.}
    \label{fig:square}
\end{figure}

\begin{figure}
    \centering
    \includegraphics[width=0.8\textwidth]{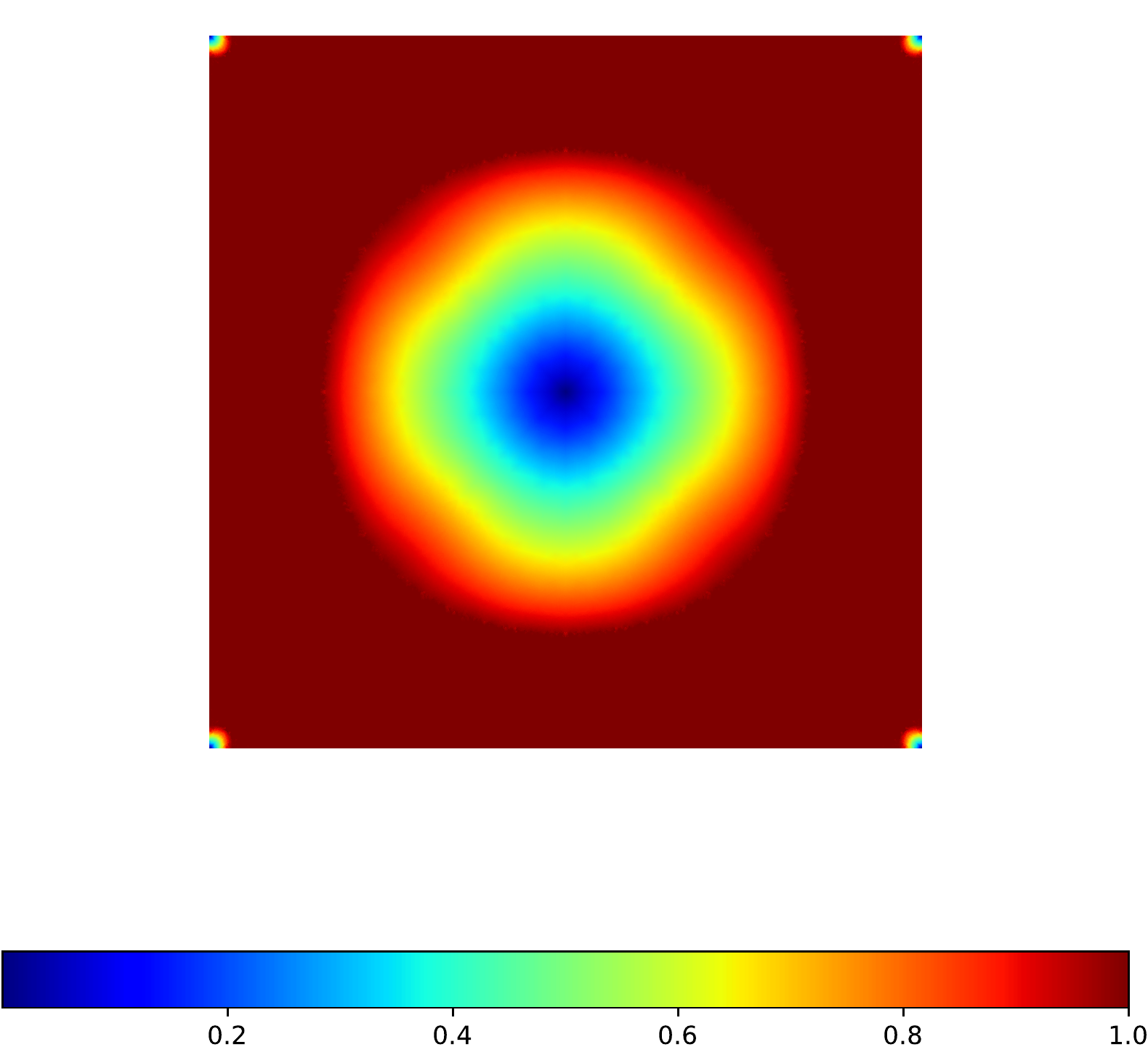}\\[1cm]
    \includegraphics[width=0.8\textwidth]{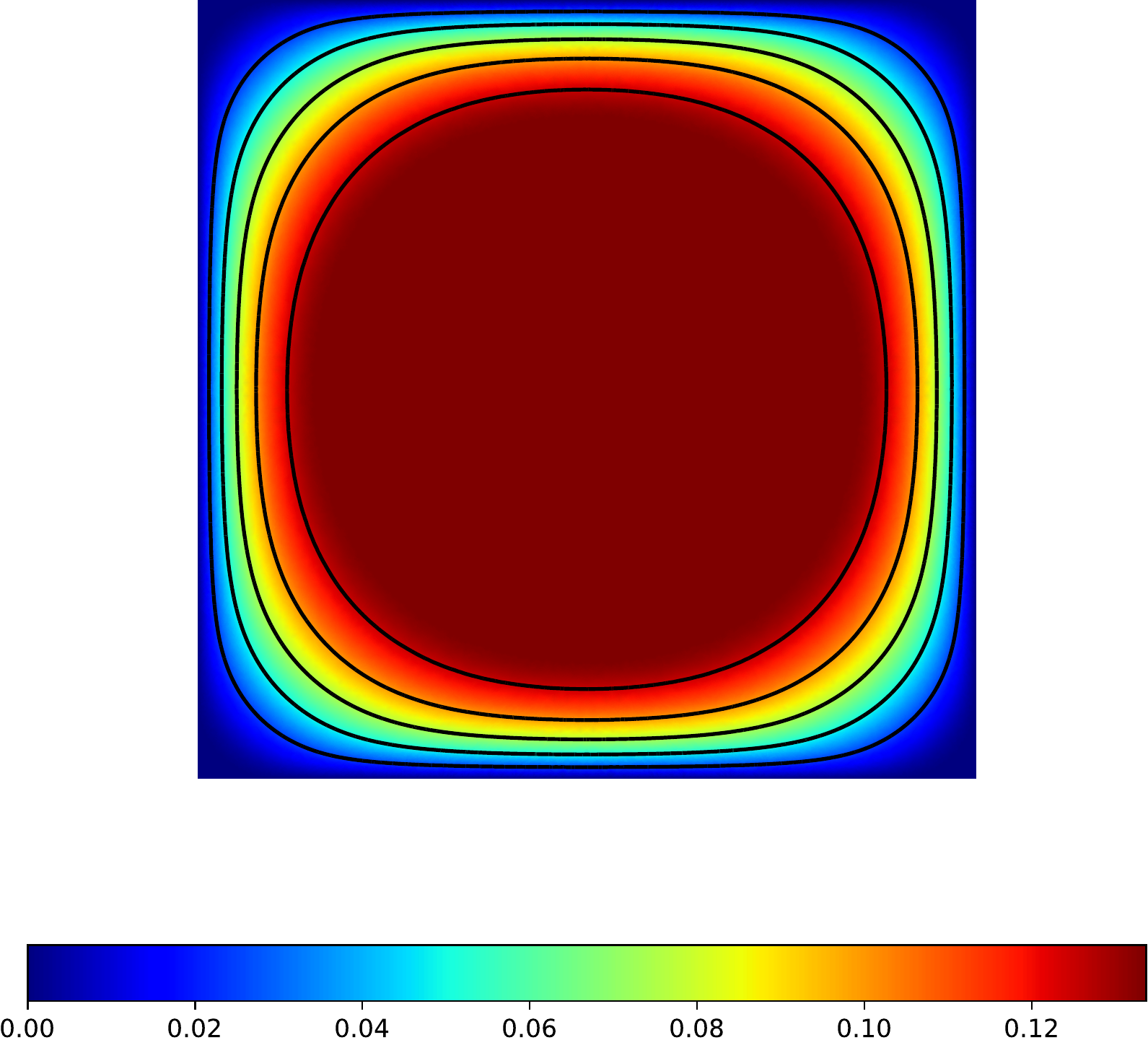}
    \caption{The length of the discrete Lagrange multiplier $|\Lam_h|$ and the discrete velocity $u_h$ for the square problem.}
    \label{fig:squareconv}
\end{figure}

\bibliographystyle{siam}
\bibliography{literature}

\end{document}